\documentclass[11pt,aap]{imsart}
\usepackage{amsmath,amssymb,amsthm,amsfonts,amstext,amsbsy,amscd}
\usepackage{color}
\usepackage{graphicx}
\DeclareMathOperator{\N}{{\mathbb N}}
\DeclareMathOperator{\R}{{\mathbb R}}
\DeclareMathOperator{\PP}{{\mathbb P}}

\newcommand{\EE}[1]{{\mathbb E}\left[#1\right]}

\newcommand{\norme}[1]{\left\| #1\right\|}
\newcommand{\block}{\mathbb{B}}
\newcommand{\abs}[1]{\left| #1\right|}

\def\b{\begin{equation}}
\def\e{\end{equation}}

\newtheorem{theo}{Theorem}

\newtheorem{lem}{Lemma}

\newtheorem{rem}{Remark}

\newtheorem{prop}{Proposition}

\def\S{{\cal S}_{\mu}}

\def\mS{{{\cal S}_{\mus}}}

\def\Sw{{\cal S}_{\mu,g}}

\def\mSw{{\cal S}_{\mus,g}}

\def \mus{{\tilde \mu}}

\def \bl{\left\langle }
\def \br{ \right\rangle }
\newcommand{\BB}[2]{{\bl #1 , #2 \br}}

\begin{document}

\begin{frontmatter}
\title{Multifractal analysis in a mixed asymptotic
framework}

\begin{aug}
\author{\fnms{Emmanuel} \snm{Bacry}
\ead[label=e1]{emmanuel.bacry@polytechnique.fr}},
\author{\fnms{Arnaud}
\snm{Gloter}\ead[label=e2]{arnaud.gloter@univ-mlv.fr}},
\author{\fnms{Marc} \snm{Hoffmann}
\ead[label=e3]{marc.hoffmann@univ-mlv.fr}} \and
\author{\fnms{Jean Fran\c{c}ois} \snm{Muzy}  \ead[label=e4]{muzy@univ-corse.fr}}


\affiliation{Ecole Polytechnique, Universit\'e Paris-Est and
Universit\'e de Corse}

\address{   Centre de Math\'ematiques appliqu\'ees, \\
  Ecole polytechnique, \\
  91128 Palaiseau, France \\
\printead{e1}}

\address{
Laboratoire d'Analyse et de Math\'ematiques Appliqu\'ees, \\
  Universit\'e Paris-Est, \\
  5, Bd Descartes, Champs-sur-Marne, France \\
\printead{e2}\\ \printead{e3}}

\address{
CNRS UMR 6134,
  Universit\'e de Corse,
  \\
  Quartier Grosseti, 20250, Corte, France\\
\printead{e4}}
\end{aug}

\begin{abstract}
Multifractal analysis of multiplicative random cascades is
revisited within the framework of {\em mixed asymptotics}. In this
new framework, statistics are estimated over a sample which size
increases as the resolution scale (or the sampling period) becomes
finer. This allows one to continuously interpolate between the
situation where one studies a single cascade sample at arbitrary
fine scales and where at fixed scale, the sample length (number of
cascades realizations) becomes infinite. We show that scaling
exponents of ''mixed'' partitions functions i.e., the estimator of
the cumulant generating function of the cascade generator distribution,
depends on some ``mixed asymptotic'' exponent $\chi$ respectively
above and beyond two critical value $p_\chi^-$ and $p_\chi^+$.
We study the convergence properties of partition functions in
mixed asymtotics regime and establish a central limit theorem.
These results are shown to remain valid within a general wavelet analysis
framework. Their interpretation in terms of Besov frontier are discussed.
Moreover, within the mixed asymptotic framework, 
we establish a ``box-counting'' multifractal formalism  that can be seen
as a rigorous formulation of Mandelbrot's negative dimension theory.
Numerical illustrations of our purpose on specific examples are
also provided.
\end{abstract}

\end{frontmatter}

\section{Introduction}

Multifractal processes have been used successfully in many
applications which involve series with invariance scaling
properties. Well known examples are fully developed turbulence where
such processes are used to model the velocity or the
dissipation energy fields \cite{Fri95} or finance, where they have been
shown to reproduce very accurately the major ``stylized facts'' of
return time-series \cite{BouPot03,CalFisMan97,MuzDelBac00}. Since
pioneering works of Mandelbrot \cite{Man74a,Man74b}, Kahane and
Peyri\`ere \cite{KahPey76}, a lot of mathematical studies have been
devoted to multiplicative cascades, denoted in sequel as ${\cal
M}$-cascades (see e.g. refs \cite{Gui90,Liu02,Mol96,OssWay00}). One
of the central issues of these studies was to understand how the
partition function scaling exponents (hereafter denoted as
$\tau_0(q)$), are related, on one hand, to the cumulant generating
function of cascade weight distribution and, on the other hand, to
the regularity properties of cascade samples. Actually, the goal of the multifractal formalism
is to directly relate the function $\tau_0(q)$ to the so-called singularity spectrum, i.e., the Hausdorff dimension of the set of all
the points corresponding to given H\"older exponent.
Let us mention that
recently continuous versions of multiplicative cascades have been
introduced \cite{BacMuz03,BarMan02} : they share most of properties
with discrete cascades but do not involve any preferential scale
ratio and remain invariant under time translation. In these
constructions, the analog of the integral scale $T$, i.e., the
coarsest scale where the cascade iteration begins, is a correlation
time.

In all the above cited references, the main results concern one
single cascade over one integral scale $T$ in the limit of arbitrary
small sampling scale. However, in many applications (e.g., the
above turbulence experiments) there is no reason a priori that the
length of the experimental series corresponds to one (or few)
integral scale(s). From a general point of view, as long as modeling
a discrete (time or space) series with a cascade process is
concerned, three scales are involved : (i) the resolution scale $l$
which corresponds to the sampling period of the series, (ii) the
{\em integral (or correlation) scale} $T$ and (iii) the size $L$ of
the whole series. Using these notations,  the total number of
samples of the series is
\begin{equation*} 
N = \frac L l.
\end{equation*}
Therefore, when modeling a discrete series with a multifractal
process, various types of asymptotics for $N\rightarrow +\infty$ can
be defined. The ``high resolution asymptotics'' considered in the
literature, corresponds to $l \rightarrow 0$ whereas $L$ is fixed.
On the other side, one could also consider the ``infinite historic
asymptotics'' that corresponds to $L\rightarrow +\infty$ whereas $l$
is fixed. If we define $N_{T}$ to be the number of integral scales
involved in the series
\begin{equation*}
N_{T} = \frac L T,
\end{equation*}
and $N_{l}$ the
number of samples per integral scale
\begin{equation} \label{eq:Nl}
N_{l} = \frac  T l,
\end{equation}
then we have
\begin{equation*}
 N =
N_{T} N_{l}.
\end{equation*}
Thus, the high resolution asymptotics corresponds to $N_{T}$ fixed
and $N_{l}\rightarrow +\infty$ whereas the infinite historic
asymptotics corresponds to $N_{l}$ fixed and $N_{T}\rightarrow
+\infty$. But in many applications, it is clear that since the
relative values of $N_{T}$ and $N_{l}$ can be arbitrary, it is not
obvious that one of the two mentionned asymptotics can account
suitably for situation. This leads us to consider an asymptotics
according to which $N_{T}$ and $N_{l}$ go to infinity (and therefore
$N$ goes to infinity) and at the same time preserve their relative
``velocities'', i.e., the ratio of their logarithm. Some of us, have
already suggested the following ``mixed asymptotics''
\cite{Koz06,MuzBaBaPo07,MuzBacKoz06}~:
\begin{equation*}
 N_{T} = N_{l}^\chi,
\end{equation*}
where $\chi \in \R^+$ is a fixed  number that quantifies the
relative velocities of  $N_{T}$ and $N_{l}$. Thus,
\begin{itemize}
\item $\chi = 0$ corresponds to the  high resolution asymptotics,
\item $\chi \rightarrow +\infty$ corresponds to the  infinite historic asymptotics,
\end{itemize}
and all other values are truly ``mixed'' asymptotics. Successful
applications of the mixed asymptotics have already been performed
\cite{MuzBaBaPo07,MuzBacKoz06}. In this paper we revisit the
standard problems of (i) the estimation of cascade generator
cumulant generating function in the mixed asymptotic framework and
of (ii) the multifractal formalism or of how to relate this function
to a dimension-like quantity.

The paper is organized as follows: in section \ref{Sec:mcdef} we
recall basic definitions and properties of ${\cal M}$-cascades.
Section \ref{Sec:mas} contains the main results of this paper. If we
define a multifractal measure $\tilde{\mu}$ as the concatenation of
independent $\cal{M}$-cascades of length $T$, with common generator
law $W$, then we show in Theorem \ref{th:MSscaling}:
$$
\frac{1}{\log{(1/l)}} \log \left( N^{-1} \sum_{k=0}^{N-1}
\tilde{\mu} \left( [kl,(k+1)l] \right)^p \right) \rightarrow p -
\log_2 E[W^p]:=\tau(p)+1
$$
for $p$ in some range $(p_\chi^-,p_\chi^+)$. These critical
exponents $p_\chi^-$, $p_\chi^+$ are related to the two solutions,
$h_\chi^-$, $h_\chi^+$ of the equation $D(h)=-\chi$ where
\begin{equation*}
\nonumber
D(h)=\inf_p \{ph - \tau(p) \}
\end{equation*}
is the Legendre transform of $\tau$. The convergence rate is studied in Section \ref{TCL}.
 Let us stress that the range of
validity on $p$ of this convergence  is wider in the mixed
asymptotic framework ($\chi>0$) than in the high resolution asymptotic
($\chi=0$). As a consequence we can relate $D(h)$ to a "box-counting dimension" (sometimes referred to as a box dimension \cite{Jaf04}
or a coarse-grain spectrum \cite{Rie99}), and
derive, as stated in  Theorem
\ref{T:formalism_D_neg}, a "box-counting multifractal formalism" for  $\tilde{\mu}$
$$
\frac{1}{N_T} \# \{k\in\{0,\dots,N\} \mid \tilde{\mu} \left(
[kl,(k+1)l] \right) \in [l^{h-\varepsilon},l^{h+\varepsilon}] \}
\simeq l^{-D(h)}
$$
in the range of $[h_\chi^+,h_\chi^-]$.  
Since for $\chi > 0$, $D(h)$ can take negative values in previous equation, 
this can be seen as a rigorous formulation of Mandelbrot's negative dimension theory \cite{Man90a,Man03}.
%
In section \ref{Sec:wavelets}, we extend previous results to
partition functions relying on some arbitrary wavelet decomposition
of the process. In Section \ref{S:Link_with} we give an interpretation
of the results connected with the Besov frontier associated with our
multifractal measure.
Finally, in section \ref{Sec:num} we discuss some specific examples where the
law of the cascade generator is respectively log-normal, log-Poisson
and log-Gamma. For illustration purpose, we also report, in each
case, estimations performed from numerical simulations.
Auxiliary useful Lemmas are moved to Appendices.

\section{ ${\cal M}$-cascades : Definitions and properties}
\label{Sec:mcdef}
\subsection{Definition of the ${\cal M}$-cascades}
\label{sec:mcascades} Let us first introduce some  notations. Given
a $j$-uplet $r=(r_{1},\ldots,r_{j})$, for all strictly positive
integer $i \le j$, we note $r|i$ the restriction of the $j$-uplet to
its first $i$ components, i.e.,
\begin{equation*} r|i =
(r_{1},\ldots,r_{i}),~~\forall i \in \{1,\dots,j\}.
\end{equation*}
By convention, if $j = 0$, we consider that $r = \emptyset$ and in
the sequel, we denote by $rr'$ the $j+j'$-uplet obtained by
concatenation of $r\in \{0,1\}^j$ and $r'\in \{0,1\}^{j'}$.
Moreover, we note
\begin{equation*} \overline r =
\begin{cases}
 2^j\sum_{i=1}^{j} r_{i}2^{-i},~~\mbox{  if } r \neq \emptyset \\
 0~~~~~~~~~~~~~~~~~~\mbox{  if } r = \emptyset
 \end{cases}.
\end{equation*}
Let fix $T \in \R^{+*}$ and $k\in\N$. We define $I_{j,k}$ as the
interval
\begin{equation} \label{eq:achanger} I_{j,k} =
[k2^{-j}T,(k+1)2^{-j}T].
\end{equation}
Thus, for any $j\in \N^*$,
the interval $[0,T]$ can be decomposed as $2^j$ dyadic intervals :
\begin{equation*}
[0,T] = \bigcup_{r\in\{0,1\}^j} I_{j,\overline r}.
\end{equation*}
Let us now build the so called ${\cal{M}}$-cascade measures
introduced by Mandelbrot in 1974 \cite{Man74b}. Let
$\{W_{r}\}_{r\in\{0,1\}^{j},~j\in\N^*}$ be a set of i.i.d random
variables of mean $\EE{W_{r}} = 1$. Given $j\in\N^*$, we define the
random measure $\mu_{j}$ on $[0,T]$ such that, for all $r \in
\{0,1\}^j$, the Radon-Nikodym derivative with respect to the
Lebesgue measure $\frac{\text{d} \mu_j}{ \text{d} x}$ is constant on
$I_{j,\overline{r}}$ with:
\begin{equation} \label{eq:def_mul}
\frac{\text{d} \mu_j }{ \text{d} x}= \prod_{i=1}^j W_{r|i}, \quad
\text{on $I_{j,\overline{r}}$, for $r \in \{ 0,1\}^j$.}
\end{equation}
As it is well known
\cite{KahPey76}, the measures $\mu_{j}$ have a non-trivial limit
measure $\mu_{\infty}$, when $j$ goes to $\infty$, as soon as $
\EE{W\log_{2}W} < 1$.
 Moreover, the total mass
 $$
\mu_{\infty}\left([0,T]\right) = \lim_{j \to \infty} T 2^{-j} \sum_{
r \in \{0,1\}^j} \prod_{i=1}^{j} W_{r \mid i},
$$
 verifies $ \EE{\mu_{\infty}\left([0,T]\right) } = T$. Let us remark
 that if $r \in \{0,1 \}^j$ then by construction we have:
 \begin{align} \nonumber
\mu_\infty ( I_{j,\overline{r}}) &= \lim_{n \to \infty} T 2^{-j}
\prod_{i=1}^j W_{r \mid i}  \left( \sum_{r' \in \{0,1 \}^n} 2^{-n}
\prod_{i=1}^n W_{rr' \mid (j+i)} \right)
\\ \label{E:auto_sim_simple}
&= 2^{-j} \left( \prod_{i=1}^j W_{r \mid i} \right)
\overline{\mu}_{\infty}^{(r)} \left( [0,T] \right),
 \end{align}
where $\overline{\mu}_{\infty}^{(r)} $ is a $\mathcal{M}$-cascade
measure on $[0,T]$ based on the random variables $W_{rr'}$ for $r'
\in \cup_{j\ge1} \{0,1 \}^j$. This equality is usually referred to as
''Mandelbrot star equation''.

In the
sequel we need the following  set of assumptions:
\begin{align} \label{eq:hypo1}
 \EE{W\log_{2}W} < 1, ~ \PP(W=1)<1,
 \\
  \PP(W>0) = 1 ,~ \EE{W^p}<\infty \text{ for all $p \in \R$.}
 \label{eq:momentsW}
 \end{align}
Let $\tau(p)$ be the smooth and concave function defined on $\R$
by
\begin{equation} \label{eq:taup}
\tau(p) = p-\log_{2}\EE{W^p}-1.
\end{equation}
Let us notice that $\log_2 \EE{W^p}$ is nothing but the cumulant
generating function (log-Laplace transform) of the logarithm of
cascade generator distribution. It is shown in \cite{KahPey76} that
for $p >1$, the condition $\tau(p) >0$ implies the finiteness of
$\EE{\mu_{\infty}\left([0,T]\right)^p }$. By Theorem 4 in
\cite{Mol96}, the conditions \eqref{eq:momentsW} imply the existence
of finite negative moments $\EE {
\mu_{\infty}\left([0,T]\right)^p}$, for all $p<0$.

\subsection{Multifractal properties of ${\cal M}$-cascades}
A {\cal M}-cascade is a multifractal measure and the study of its
multifractal properties reduces to the study of the {\em partition
function }
\begin{equation}
\label{eq:partition}
\S(j,p) = \sum_{k=0}^{2^j-1} \mu_{\infty}(I_{j,k})^p.
\end{equation}
\label{sec:scaling} Basically, one can show \cite{Mol96,OssWay00}
that, for fixed $p$,  this partition function behaves, when $j$ goes
to $\infty$, as a power law function of the scale $|I_{j,k}| = T2^{-j}$.
More precisely,
let us introduce the two following critical exponents:
\begin{align*}
p_0^+=\inf \{p \ge 1 \mid p \tau'(p)-\tau(p) \le 0\} \in (1,\infty]
\\
p_0^-=\sup \{p \le 0 \mid p \tau'(p)-\tau(p) \le 0\} \in
[-\infty,0).
\end{align*}
If $p_0^+$ (resp. $p_0^-$) is finite we set $h_0^+=\tau'(p_0^+)$
(resp. $h_0^-=\tau'(p_0^-)$).

\begin{theo}
\label{th:Sscaling}
{\bf Scaling of the partition function} \cite{OssWay00} \\
Let $p \in  \R$, the power law scaling exponent of $\S(j,p)$ is given by
\begin{equation}
\label{eq:partitionscaling} \lim_{j \rightarrow \infty}
\frac{\log_{2}\S(j,p)}{-j}
\operatornamewithlimits{\longrightarrow}_{a.s.}  \tau_{0}(p),
\end{equation}
 where $\tau_{0}(p)$ is defined by
 \begin{equation}
\label{eq:tau0} \tau_0(p) = \left\{
\begin{array}{ll}
 \tau(p), & \forall p\in (p_0^-,p_0^+) \\
 h_0^+ p,  & \forall p\ge p_0^+ \\
 h_0^- p,  & \forall p\le p_0^-
\end{array}
\right. .
\end{equation}
\end{theo}
The proof can be found in \cite{OssWay00}. The convergence in
probability of \eqref{eq:partitionscaling} was obtained in the
earlier work \cite{Mol96}. This  theorem basically states that
$\S(j,p)$ behaves like
\begin{equation*}
\S(j,p) \simeq 2^{-j\tau_{0}(p)}.
\end{equation*}
 Let us
note that the partition  function \eqref{eq:partition} can be
rewritten in the following way
\begin{equation}
\label{eq:partition1} \S(j,p) = \sum_{r\in\{0,1\}^j}
\mu_{\infty}(I_{j,\overline r})^p
\end{equation}
 and using
(\ref{E:auto_sim_simple}), one gets
\begin{equation} \label{eq:Srep} \S(j,p) =
2^{-jp}\sum_{r\in\{0,1\}^j} \prod_{i=1}^j W_{r|i}^p
\left(\bar\mu^{(r)}_{\infty}([0,T])\right)^p,
\end{equation}
where the $\{\bar \mu^{(r)}_{\infty}([0,T])\}_{r\in\{0,1\}^j}$ are
i.i.d. random variables with the same law as $ \mu_{\infty}([0,T])$.
Thus, a simple computation shows that
\begin{equation*}
 \EE{\S(j,p)} =
2^{-jp}2^j\EE{W^p}^j \EE{\mu_{\infty}([0,T])^p} = 2^{-j\tau(p)}
\EE{\mu_{\infty}([0,T])^p}.
\end{equation*}
One sees that the last theorem states that, in the case $p\in
[p_0^-,p_{0}^+]$,   $\S(j,p)$ scales as its mean value.

On the other hand, the fact that for $p \notin [p^-_{0}, p^+_{0}]$
the partition function scales as given in \eqref{eq:tau0} instead of
scaling as its mean value, is referred to as the 'linearization
effect'. A possible explanation of this effect is that for $p$
larger than the critical exponents $p^+_0$ (resp. smaller than
$p^-_0$), the sum involved in the partition function
\eqref{eq:partition1} is dominated by its supremum (resp. infimum) term. Thus one
should not expect a law of large number to hold for the behavior of
this sum.
Another possible interpretation of this theorem in the case
$p>p_{0}^+$ is given in Section \ref{S:Link_with}.

\section{Mixed asymptotics for ${\cal M}$-cascades}
\label{Sec:mas}
\subsection{Mixed asymptotics : definitions and notations}
A convenient way to construct a multifractal measure on
$\mathbb{R}_+$, with an integral scale equal to $T$, is to patch
independent realizations of $\mathcal{M}$-cascades measures.
More precisely, consider $\{\mu_{\infty}^{(m)}\}_{m\in\N}$ a
sequence of i.i.d ${\cal M}$-cascades on $[0,T]$ as defined in
Section \ref{sec:mcascades} and define the stochastic measure on
$[0,\infty)$ by:
\begin{equation} \label{eq:mus}
\mus \left([t_{1},t_{2}]\right) = \sum_{m=0}^{+\infty}
\mu_{\infty}^{(m)}\left( [t_{1}-mT,t_{2}-mT]\right), \quad \text{
for all $0 \le t_1 \le t_2$.}
\end{equation}
This model is entirely defined as soon as both $T$ and the law of
$W$ are fixed. The discretized time model for the $N$ samples of the
series is $\{\mus[kl,(k+1)l]\}_{0\le k< N-1}$.

\subsection{Scaling properties}
In this section, we study the partition function for the measure
$\mus$ as defined in Eq. \eqref{eq:mus} in the mixed asymptotic
limit. $T$ is fixed, we choose the sampling step
\begin{equation*} l = T 2^{-j},~N_{l}
= 2^{j},
\end{equation*}
 and the number of integral scales is
related with the sampling step as
\begin{equation*}
N_{T} = \lfloor N_{l}^\chi \rfloor \sim 2^{j\chi},
\end{equation*}
with $\chi > 0$ fixed. According to \eqref{eq:Nl}, one gets for the
total number of data:
\begin{equation*}
N = N_{T}2^j \sim 2^{j(1+\chi)}.
\end{equation*}
The mixed asymptotics corresponds to the limit $j\rightarrow
+\infty$. The partition function of $\mus$ can be written as (recall
\eqref{eq:achanger}):
\begin{eqnarray}
\label{eq:partitionmix}
\mS(j,p) & = & \sum_{k=0}^{N-1}   \mus(I_{j,k})^p \\
\label{eq:partitionmix1}
& = &\sum_{m=0}^{N_{T}-1}  \S^{(m)}(j,p),
\end{eqnarray}
where $\S^{(m)}(j,p)$ is the partition function of $\mu_{\infty}^{(m)}$, i.e.,
\begin{equation}
\label{eq:Sm}
 \S^{(m)}(j,p) = \sum_{k=0}^{2^j-1} \mu_{\infty}^{(m)}(I_{j,k})^p.
 \end{equation}
Let us state the results  of this section. We introduce the two
critical exponents in the mixed asymptotic framework:
\begin{align}
p_\chi^+=\inf \{p \ge 1 \mid p \tau'(p)-\tau(p) \le -\chi \} \in
(1,\infty] \label{E:defPchiplus}
\\
p_\chi^-=\sup \{p \le 0 \mid p \tau'(p)-\tau(p) \le -\chi \} \in
[-\infty,0), \label{E:defPchimoins}
\end{align}
and we set when these critical exponents are finite
$h_\chi^+=\tau'(p_\chi^+)$, $h_\chi^-=\tau'(p_\chi^-)$.
\begin{theo}
\label{th:MSscaling}
{\bf Scaling of the partition function in a mixed asymptotics } \\
Let $p\in\R$ and $\mus$ be the random measure defined by
\eqref{eq:mus} where the law of $W$ satisfies
\eqref{eq:hypo1}--\eqref{eq:momentsW}. We assume that, either
$p_\chi^+=\infty$, or $p_\chi^+<\infty$ with $\tau(p_\chi^+)>0$.
Then, the power law scaling of $\mS(j,p)$ is given by
\begin{equation} \label{eq:Mpartitionscaling} \lim_{j \rightarrow
\infty} \frac{\log_{2}\mS(j,p)}{-j}
\operatornamewithlimits{\longrightarrow}_{a.s.} \tau_{\chi}(p),
\end{equation}
where $\tau_{\chi}(p)$ is defined by
\begin{equation*}
 \tau_\chi(p) = \left\{
\begin{array}{ll}
 \tau(p)-\chi, & \forall p\in (p_\chi^-,p_\chi^+) \\
 h_\chi^+ p,  & \forall p\ge p_\chi^+ \\
  h_\chi^- p,  & \forall p\le p_\chi^-
\end{array}
\right. .
\end{equation*}

\end{theo}
\begin{rem} \label{R:moment}
If $p_\chi^+=\infty$ then simple considerations on the concave
function $\tau$ shows that $\tau(p)>0$ for all $p >1$, and hence the
cascade measure has finite moments of any positive orders.
Otherwise the
assumption $\tau(p_\chi^+)>0$ is stated in Theorem
\ref{th:MSscaling} to insure $\EE{\mu_{\infty}([0,T])^p}<\infty$ for
$p \in [0, p_\chi^+)$. Such assumption was not needed in Theorem
\ref{th:Sscaling}, since on can check that necessarily
$\tau(p^+_0)>0$.
\end{rem}
\begin{rem} Let us stress that the behavior of the partition
function is largely affected by the choice of a mixed asymptotic:
the 'linearization effect' now occurs for $p$ in the set
$(-\infty,p_\chi^-) \cup (p_\chi^+,\infty)$, which is smaller when
$\chi$ increases.
\end{rem}

\begin{theo}
\label{th:MSup} {\bf Scaling of the supremum and the infimum
of the mass in a mixed asymptotics} \\
Assume \eqref{eq:hypo1}--\eqref{eq:momentsW}. Then, if
$p_\chi^+<\infty$, one has,
\begin{equation} \label{eq:Msupscale} \lim_{j\rightarrow +\infty}
\frac {\log_{2} \sup_{k\in[0,N-1]} \mus(I_{j,k})}{-j} = h_{\chi}^+,
\text{ almost surely,}
\end{equation}
and if $p_\chi^- > - \infty$, one has,
\begin{equation} \label{eq:Minfscale} \lim_{j\rightarrow +\infty}
\frac {\log_{2} \inf_{k\in[0,N-1]} \mus(I_{j,k})}{-j} = h_{\chi}^-,
\text{ almost surely.}
\end{equation}
\end{theo}
The theorem \ref{th:MSup} shows that when the 'linearization effect'
occurs, the scaling of the partition function
\eqref{eq:partitionmix} is governed by its supremum and infimum terms
for respectively large positive and negative $p$ values.

These theorems will be proved in three parts. In Section
\ref{sec:proofScaling}, we will prove Eq.
\eqref{eq:Mpartitionscaling} of Theorem \ref{th:MSscaling} only for
$p\in (p_{\chi}^- ,p_{\chi}^+)$. In Section \ref{sec:proofScaling2},
we will prove the case $p \notin (p_{\chi}^- ,p_{\chi}^+)$ and
Theorem \ref{th:MSup} is shown in Section \ref{sec:proofSup} to be a
simple corrolary of this last case.

\subsection{Proof of Theorem \ref{th:MSscaling}}
First we need  an auxiliary result which is helpful in the sequel.
\subsubsection{Limit theorem for a rescaled cascade} \label{S:prop_gene}
For each $m$ we denote as $(W^{(m)}_r)_{r \in \cup_{j} \{0,1\}^j}$
the set of i.i.d. random variables used for the construction of the
measure $\mu^{(m)}_\infty$. Moreover we assume that for each $m \ge
0$, $j\ge 0$, $r \in \{0,1 \}^j$ we are given a random variable
$Z^{(m,r)}$, measurable with respect to the sigma-field
$\sigma\left( W^{(m)}_{rr'} \mid r' \in \cup_{j} \{0,1\}^j \right)$.
We make the assumption that the law of $Z^{(m,r)}$ does not depend
on $(m,r)$, and denote by $Z$ a variable with this law.

Let us consider the quantities, for $p \in \R$:
\begin{equation} \label{E:def_Mj}
\mathcal{M}^{(m)}_j(p)= 2^{-jp }\sum_{r \in \{0,1 \}^j}
\prod_{i=1}^j \left( W^{(m)}_{r \mid i} \right)^p Z^{(m,r)},
\end{equation}
and
\begin{equation} \label{E:def_Nj}
\mathcal{N}_j(p)=\sum_{m=0}^{N_T-1} \mathcal{M}^{(m)}_j(p).
\end{equation}
\begin{prop} \label{P:cascade_conser_mixte}
Assume that for some $\epsilon >0$,
$\EE{\abs{Z}^{1+\epsilon}}<\infty$ and  $-p \tau'(p)+\tau(p) <
\chi$, then:
\begin{equation*}
2^{j (\tau(p)-\chi)} \mathcal{N}_j(p) \xrightarrow{j \to \infty}
\EE{Z}, \text{ almost surely.}
\end{equation*}
\end{prop}
\begin{proof}
From \eqref{E:def_Mj}--\eqref{E:def_Nj} and the definition
\eqref{eq:taup} we get,
\begin{align*}
\EE{\mathcal{N}_j(p)} & =  N_T 2^{j} 2^{-jp} \EE{W^p}^j \EE{Z} \\
& \displaystyle \sim_{j \to \infty} 2^{j\chi} 2^{-j\tau(p)} \EE{Z}.
\end{align*}
Hence the proposition will be proved if we show:
\begin{equation} \label{E:notre_but_cv}
2^{j (\tau(p)-\chi)} \left( \mathcal{N}_j(p)- \EE{\mathcal{N}_j(p)}
\right) \xrightarrow{j \to \infty} 0, \text{ almost surely.}
\end{equation}
For an arbitrary small $\epsilon >0$, we study the
$L^{1+\epsilon}(\PP)$ norm of the difference. Set,
\begin{equation}
\label{eq:step1} L_{N}^{1+\epsilon} = \EE{|\mathcal{N}_j(p) -
\EE{\mathcal{N}_j(p)}|^{1+\epsilon}}.
\end{equation}
Applying successively lemmas \ref{lem:1} and \ref{lem:2} of Appendix
\ref{app:lem12}, we get:
\begin{align*}
L_{N}^{1+\epsilon} &\le C 2^{j \chi } \EE{ \abs{
\mathcal{M}^{(0)}_j(p) }^{1+\epsilon} } \\
& \le C 2^{-j[(1+\epsilon)\tau(p)-\chi]} \sum_{k=0}^{j}
2^{-k\tau(p(1+\epsilon))} 2^{k(1+\epsilon)\tau(p)}.
\end{align*}
We deduce that $ 2^{j (\tau(p)-\chi)(1+\epsilon)}L_{N}^{1+\epsilon}$
is bounded by the quantity:
\begin{equation*}
C2^{-j\chi\epsilon} \sum_{k=0}^{j} 2^{-k\tau(p(1+\epsilon))}
2^{k(1+\epsilon)\tau(p)}.
\end{equation*}
Clearly, as soon as  $ 2^{-\chi\epsilon}2^{-\tau(p(1+\epsilon))}
2^{(1+\epsilon)\tau(p)} < 1$, this quantity is, in turn, bounded by
$C2^{-j\chi\epsilon'}$ for some $\epsilon'>0$. Taking the log, a
sufficient condition is
\begin{equation*} \frac {\tau(p)(1+\epsilon)
-\tau(p(1+\epsilon))}{\epsilon} < \chi
\end{equation*}
which is implied for $\epsilon$ small enough by
\begin{equation*} -p\tau'(p)+\tau(p)
< \chi.
\end{equation*}
Thus we have shown that $2^{j
(\tau(p)-\chi)(1+\epsilon)}L_{N}^{1+\epsilon}$ is asymptotically
smaller than $2^{-j \epsilon'}$ with some $\epsilon'>0$.
Using the Bienaym\'e-Chebyshev inequality leads to
\begin{equation*}
\PP\{ 2^{j (\tau(p)-\chi)} |  \mathcal{N}_j(p)-
\EE{\mathcal{N}_j(p)}| \geq
\eta \} \le\frac{2^{j (\tau(p)-\chi)(1+\epsilon)}L_{N}^{1+\epsilon}}{\eta^{1+\epsilon}} \le \frac {C 2^{-j \epsilon'}}{\eta^{1+\epsilon}}
\end{equation*}
 for any $\eta >0$.
 A simple use of the Borel Cantelli lemma
shows \eqref{E:notre_but_cv}.
\end{proof}
\subsubsection{Proof of Theorem \ref{th:MSscaling} for $p\in (p_{\chi}^-,p_{\chi}^+)$}
\label{sec:proofScaling}
From \eqref{eq:partitionmix1}--\eqref{eq:Sm} and the representation
\eqref{eq:Srep} for the partition function of a single cascade, we
see that that $\mS(j,p)$ exactly has the same structure as the
quantity $\mathcal{N}_j(p)$ of section \ref{S:prop_gene} where
$Z^{(m,r)}= \overline{\mu}_{\infty}^{(m,r)}([0,T])^{p}$ are random
variables distributed as $Z=\mu_{\infty}([0,T])^{p}$.

By definition (recall
\eqref{E:defPchiplus}--\eqref{E:defPchimoins}), the condition $-p
\tau'(p)+\tau(p) < \chi$ holds for any $p \in (p_\chi^-,p_\chi^+)$,
and by Remark \ref{R:moment}, $\EE{|Z|^{1+\epsilon}}<\infty$ for
$\epsilon$ small enough.

Thus, an application of Proposition \ref{P:cascade_conser_mixte}
yields the almost sure convergence:
\begin{equation} \label{E:cv_partition_renor}
2^{j\tau_\chi(p)} \mS(j,p)= 2^{j(\tau(p)-\chi)} \mS(j,p)
\xrightarrow{j \to \infty} \EE{ \mu_{\infty}([0,T])^p }.
\end{equation}
This proves the theorem \ref{th:MSscaling} for the case $p \in
(p_\chi^-,p_\chi^+)$.
\subsubsection{Proof of Theorem \ref{th:MSscaling} for $p\notin (p_{\chi}^-,p_{\chi}^+)$}
\label{sec:proofScaling2} The following proof is an adaptation of
the corresponding proof in \cite{ResSamGilWil03}.
 We need the following notations:
\begin{align*}
\mS^*(j) &= \sup_{k\in[0,N-1]} \mus_{\infty}(
[k2^{-j}T,(k+1)2^{-j}T]),
\\
m_{sup}(p) &= \limsup_{j\rightarrow \infty} \frac{\log_2
\mS(j,p)}{-j}, \quad m_{inf}(p) = \liminf_{j\rightarrow \infty}
\frac{\log_2 \mS(j,p)}{-j},
\\
 m_{sup}^* &= \limsup_{j\rightarrow \infty}
\frac{\log_2 \mS(j)^*}{-j}, \quad m_{inf}^* = \liminf_{j\rightarrow
\infty} \frac{\log_2 \mS(j)^*}{-j}.
\end{align*}


In Section \ref{sec:proofScaling} we proved that for all $p\in
(p_{\chi}^-,p_{\chi}^+)$ the following holds almost surely:
\begin{equation*}
 m_{sup}(p) = m_{inf}(p) = \tau_{\chi}(p).
\end{equation*}
We may assume that on a event of probability one, this equality
holds for all $p$ in a countable and dense subset of
$(p_{\chi}^-,p_{\chi}^+)$.

From the sub-additivity of $x \mapsto x^\rho$,
\begin{equation*} \forall \rho \in ]0,1[,~\forall p\in \R~~\mS(p,j)^\rho \le
\mS(\rho p,j), \end{equation*}
 and thus
 \begin{equation*} m_{inf}(p)
\ge \frac {m_{inf}(\rho p)} \rho.
\end{equation*}
But we have seen that $m_{inf}(\rho p) = \tau_\chi(\rho p)$, for a
dense subset of $\rho p \in (p_\chi^-,p_\chi^+)$. Assume now for
simplicity that $p\ge p_\chi^+$ and let $\rho  \rightarrow
(p_\chi^+/p)$, we get
\begin{equation} \label{eq:interm7} \forall p\ge p_\chi^+,~~\frac
{m_{inf}(p)} p \ge \frac {\tau_\chi(p_\chi^+)}{p_\chi^+} =
\frac{\tau(p_\chi^+)-\chi}{p_\chi^+} = h_\chi^+,
\end{equation} where
we have used \eqref{E:defPchiplus}

On the other hand, let $p>0$, $q \in [0,p_\chi^+)$, and $q'  \in
[0,q)$, we have
\begin{eqnarray*}
\mS(j,q)  &=&\sum_{k=0}^{N-1} \mus_{\infty}([k2^{-j}T,(k+1)2^{-j}T])^q \\
& \le & \mS^*(j)^{q-q'}\mS(j,q')  \\
& \le & \mS(j,p)^{\frac{q-q'}{p}}\mS(j,q').
\end{eqnarray*}
 Thus
\begin{equation*} m_{sup}(q)  \ge (q-q') \frac {m_{sup}(p)} p + m_{inf}(q'), \end{equation*}
then \begin{equation*} \frac {m_{sup}(p)} p \le \frac
{m_{sup}(q)-m_{inf}(q')}{q-q'} =  \frac
{\tau_{\chi}(q)-\tau_{\chi}(q')}{q-q'}. \end{equation*}
 Taking the
limit $q'\rightarrow q^-$
\begin{equation*} \frac {m_{sup}(p)} p \le
\inf_{q\in[0,p_\chi^+)} \tau'_{\chi}(q) \le \tau'_{\chi}(p_{\chi}) =
h_\chi^+. \end{equation*}
 Merging this last relation with
\eqref{eq:interm7} leads to
\begin{equation} \label{eq:interm8}
\forall p \ge p_\chi^+,~~h_\chi^+\le \frac {m_{inf}(p)} p \le \frac
{m_{sup}(p)} p \le h_\chi^+,
\end{equation}
 which proves Theorem
\ref{th:MSscaling} for $p\in [p_{\chi}^+,+\infty[$. The proof for $p
\le p_\chi^-$ is similar.

\subsubsection{Proof of Theorem \ref{th:MSup}}
\label{sec:proofSup} The following proof is an adaptation of the
corresponding proof in \cite{ResSamGilWil03}.  We have for $p>0$,
\begin{equation*} \mS^*(j)^p \le \mS(j,p) \le N \mS*(j)^p =
\lfloor 2^{j\chi} \rfloor 2^{j} \mS^*(j)^p,
\end{equation*}
 thus
\begin{equation*} p m_{inf,sup}^* \ge m_{inf,sup}(p) \ge 1+\chi + p
m_{inf,sup}^*,
\end{equation*}
which means that \begin{equation*}
 \frac {m_{inf,sup}(p)} p - \frac{1+\chi} p \ge  m_{inf,sup}^* \ge \frac {m_{inf,sup}(p)} p,
\end{equation*}
and taking the limit $p\rightarrow +\infty$ and using
\eqref{eq:interm8} proves that
\begin{equation*} m_{sup}^* = m_{inf}^* =
h_\chi^+, \end{equation*}
 which proves \eqref{eq:Msupscale}. The proof of
 \eqref{eq:Minfscale} is obtained analogously by considering $p<0$.



\subsection{Multifractal formalism and ``negative dimensions''}
\label{S:Multifractal_formalism}
Let $D(h)$ be the Legendre transform of $\tau(p)$ :
\begin{equation*}
D(h) = \min_{p}(ph-\tau(p)),
\end{equation*}
The {\em multifractal formalism} \cite{FriPar85} gives an
interesting interpretation of $D(h)$, as soon as $D(h)>0$, in terms
of dimension of set of points with the same regularity. For ${\cal
M}$-cascades, this formalism holds \cite{Mol96}, i.e., $D(h)$
corresponds to the Hausdorff dimension of the points $t \in [0,T]$
around which $\mu_{\infty}$ scales with the exponent $h$ :
\begin{equation}
\label{eq:D0hDim} D(h) = dim_{H}\left\{t,~
\limsup_{\epsilon\rightarrow 0} \frac
{\log_{2}\mu_{\infty}([t-\epsilon,t+\epsilon])}{\log_{2}(\epsilon)}
= h\right\}.
\end{equation}
The r.h.s. of \eqref{eq:D0hDim} is usually referred to as the {\em
singularity spectrum} and therefore the multifractal formalism
simply states that $D(h)$ can be identified with the singularity
spectrum of the cascade.

In a mixed asymptotic framework,
our next result shows that some kind of multifractal formalism still
holds for $D(h)<0$ in the sense that $D(h)$ governs the behavior of
the population histogram {\em per sample} of measure values at scale
$2^{-j}$ as estimated over $2^{j\chi}$ cascade samples.
In other words, $D(h)$ coincides with a box-counting dimension
 (sometimes referred to as a box dimension \cite{Jaf04}
or a coarse-grain spectrum \cite{Rie99}).
Hence the Legendre transform of $\tau(p)$ can be interpreted as a
''population'' dimension even for singularity values above and below
$h_0^+$ and $h_0^-$. Since for these values, one has $D(h) < 0$ they
have been called  ''negative dimensions'' by Mandelbrot
\cite{Man90a}. This simply means that they cannot be observed on a
single cascade sample but one needs at least $2^{j\chi}$
realizations to observe them with a ''cardinality'' like
$2^{j(\chi+D(h))}$. In that respect, they have also been referred to
as ''latent'' singularities  \cite{Man03}.

\begin{theo} \label{T:formalism_D_neg}
Assume $p_\chi^+<\infty$, $p_\chi^->-\infty$ and $\tau(p_\chi^+)>0$.
Let $h \in (h_\chi^+,h_\chi^-)$, then:
\begin{multline}
\label{E:res_inf_card} \lim_{\varepsilon \to 0}
\mathop{\underline{\lim}}_j \frac{1}{j} \log \# \left\{ k \in
\{0,\dots,N-1\} \mid 2^{-j(h+\varepsilon)} \le \tilde{\mu}(I_{j,k})
\le 2^{-j(h-\varepsilon)}\right\} \\ = \chi + D(h),
\end{multline}
\begin{multline}
\label{E:res_sup_card} \lim_{\varepsilon \to 0} \mathop{
\overline{\lim} }_{j} \frac{1}{j} \log \# \left\{ k \in
\{0,\dots,N-1\} \mid 2^{-j(h+\varepsilon)} \le \tilde{\mu}(I_{j,k})
\le 2^{-j(h-\varepsilon)}\right\} \\ = \chi + D(h).
\end{multline}
\end{theo}
\begin{proof} {\bf $1^{\text{st}}$ step:} We focus on the cases that
yield to negative dimensions, i.e. $h \in (h_\chi^+,h_0^+) \cup
(h_{0}^-,h_{\chi}^+)$. For simplicity assume $h \in
(h_{\chi}^+,h_0^+)= (\tau'(p_\chi^+), \tau'(p_0^+))$. We can write
$h=\tau'(p)$ for some $p \in (p_0^+,p_\chi^+)$ and if we define
$$\chi'=\tau(p)-p \tau'(p) \in (0,\chi)
$$
we easily get that $h=\tau'(p_{\chi'}^+)$ and $D(h)=-\chi'$. Thus
the theorem amounts to assess the magnitude of
$$
\# \left\{ k \in \{0,\dots,N-1\} \mid 2^{-j(\tau'(p_{\chi'}^+)
+\varepsilon)} \le \tilde{\mu}(I_{j,k})  \le
2^{-j(\tau'(p_{\chi'}^+) -\varepsilon)}\right\}
$$
as $\simeq 2^{j(\chi-\chi')}$.

First we derive a lower bound for this cardinality. The idea is to
split the data into blocks of size $2^{j} \left\lfloor 2^{j\chi'}
\right\rfloor$ and rely on the behavior of the supremum of
$\tilde{\mu} (I_{j,k})$ under mixed asymptotic with index $\chi'$.
More precisely let $N'=2^{j} \left\lfloor 2^{j\chi'} \right\rfloor$
and define the blocks
\begin{equation*}
\block_a =\{a N',\dots,(a+1)N'-1 \},\quad \text{ for $a=0,\dots,
M-1:= \left\lfloor N/N' \right\rfloor -1$}.
\end{equation*}
Fix $a \in \{0,\dots,M-1 \}$, then for any $p_1<p_2<p_{\chi'}^+$ we
have
\begin{equation}\label{E:mino_sup}
\sup_{k \in \block_a} \tilde{\mu}(I_{j,k}) ^{p_2-p_1} \ge
\frac{\sum_{k \in \block_a } \tilde{\mu}(I_{j,k})^{p_2} }{\sum_{k
\in \block_a } \tilde{\mu}(I_{j,k})^{p_1}} :=
2^{j(\tau_{\chi'}(p_1)-\tau_{\chi'}(p_2))}Q_j^{(a)}(p_1,p_2),
\end{equation}
where
$$Q_j^{(a)}(p_1,p_2)=
\frac{\sum_{k \in \block_a } \tilde{\mu}(I_{j,k})^{p_2}
2^{j\tau_{\chi'}(p_2)}}{\sum_{k \in \block_a }
\tilde{\mu}(I_{j,k})^{p_1} 2^{j\tau_{\chi'}(p_1)}}.$$ Clearly the
law of $Q_j^{(a)}(p_1,p_2)$ does not depend on $a$, and
$Q_j^{(0)}(p_1,p_2)$ is the ratio of two rescaled partition
functions in mixed asymptotic with index $\chi'$. Hence by
\eqref{E:cv_partition_renor}, $Q_j^{(0)}(p_1,p_2)$ converges almost
surely to the non zero constant $\EE{\mu_{\infty}([0,T])^{p_2}} /
\EE{\mu_{\infty}([0,T])^{p_1}}$.
 We deduce that for all fixed $a \le M-1$, and $p_1, p_2 \in (0,
p_{\chi'}^+)$, the sequence $(1/Q_j^{(a)}(p_1,p_2))_{j \ge 0}$ is
bounded in probability. Write now, by \eqref{E:mino_sup}, and
$\tau_{\chi'}(p)=\tau(p)-\chi'$ for $p<p_{\chi'}^+$,
$$
\sup_{k \in \block_a}
 \tilde{\mu}(I_{j,k}) 2^{j(\tau'(p_{\chi'}^+) +\epsilon)}\ge Q_j^{(a)}(p_1,p_2) 2^{-j[ \frac{\tau(p_2)-\tau(p_1)}{p_2-p_2}-\tau'(p_{\chi'}^+)-\epsilon  ]}
$$
and choose  $p_1$, $p_2$ fixed but close enough to $p_{\chi'}^+$.
This yields, for all $j\ge0$:
$$
\sup_{k \in \block_a}
 \tilde{\mu}(I_{j,k}) 2^{j(\tau'(p_{\chi'}^+)+\epsilon)}\ge Q_j^{(a)}(p_1,p_2) 2^{j\epsilon /2} .
$$
Then using that $1/Q_j^{(a)}(p_1,p_2)$ is bounded in probability, we
get $\PP ( \sup_{k \in \block_a}
 \tilde{\mu}(I_{j,k}) \ge 2^{-j(\tau'(p_{\chi'}^+) + \varepsilon) } ) \ge
\PP ( 1/Q_j^{(a)} (p_1,p_2) \le  2^{j\epsilon /2} ) \ge 1/2$ for $j$
large enough.

Remark now that the cardinality of the set $\{ k \in \{0,\dots,N-1\}
\mid \tilde{\mu}(I_{j,k})   \ge 2^{-j(\tau'(p_{\chi'}^+)
+\varepsilon)} \} $ is immediately lower bounded by the sum
$$
\sum_{a=0}^{M-1} 1_{ \left\{ \sup_{k \in \block_a}
 \tilde{\mu}(I_{j,k}) \ge 2^{-j(\tau'(p_{\chi'}^+) + \varepsilon)} \right\}}
$$
of i.i.d. Bernoulli variables with parameter greater than $1/2$.
Then it is easily deduced, using the Borel Cantelli lemma and
$M\sim_{j \to \infty} 2^{(\chi-\chi')j}$ that with probability one:
\begin{equation} \label{E:cardinal_inf}
 \mathop{\underline{\lim}}_j  2^{-(\chi-\chi')j} \# \left\{ k \in \{0,\dots,N-1\} \mid \tilde{\mu}(I_{j,k})   \ge
2^{-j(\tau'(p_{\chi'}^+) +\varepsilon)} \right\} \ge 1/4.
\end{equation}

We now focus on upper bounds for the cardinality of the set
\begin{equation*}
\{ k \in \{0,\dots,N-1\} \mid \tilde{\mu}(I_{j,k})   \ge
2^{-j(\tau'(p_{\chi'}^+) +\eta)} \}
\end{equation*}
where $\eta$ is some real number in a neighborhood of zero.
 It is simply derived from the connection with the partition function that for any $p>0$ this cardinality is lower than
$ \mS(j,p) 2^{j (\tau'(p_{\chi'}^+) +\eta)p } $. Applying this with
$p=p_{\chi'}^+$ and since, by \eqref{E:defPchiplus},
$\tau_{\chi}(p_{\chi'}^+)=\tau(p_{\chi'}^+)-\chi=
p\tau'(p_{\chi'}^+)+\chi'-\chi$ we get the following upper bound,
\begin{multline}
\label{E:upper_bound_card} \# \{ k \in \{0,\dots,N-1\} \mid
\tilde{\mu}(I_{j,k})   \ge 2^{-j(\tau'(p_{\chi'}^+) +\eta)} \} \\
\le \mS(j,p_{\chi'}^+) 2^{j \tau_{\chi}(p_{\chi'}^+)}
2^{j(\chi-\chi'+\eta p_{\chi'}^+)}.
\end{multline}
By $p_{\chi'}^{+}<p_{\chi}^+$, the convergence result
\eqref{E:cv_partition_renor} with $p=p_{\chi'}^+$ applies and we
deduce for $\eta=-\epsilon<0$ that:
\begin{equation} \label{E:cardinal_sup}
\mathop{\overline{\lim}}_j 2^{-(\chi-\chi')j} \#  \{ k \in
\{0,\dots,N-1\} \mid  \tilde{\mu}(I_{j,k})   \ge
2^{-j(\tau'(p_{\chi'}^+) -\varepsilon)} \} =0 .
\end{equation}
Then \eqref{E:res_inf_card} is a consequence of
\eqref{E:cardinal_inf} and \eqref{E:cardinal_sup}.

Finally, the upper bound  \eqref{E:res_sup_card} is directly
obtained by applying \eqref{E:upper_bound_card}
with $\eta=\varepsilon$.

{\bf $2^{\text{nd}}$ step:} We now deal with the more classical case
$h \in [h_0^+,h_0^-]$. It is known, from the multifractal formalism
for a single {\cal M}-cascade on $[0,T]$ (see
\cite{BarSeu07,Rie99}), that
 with probability $1$:
\begin{multline}\label{E:formalism_une}
\lim_{\varepsilon \to 0} \mathop{\underline{\lim}}_{j} \frac{1}{j}
\log \# \left\{ k \in \{0,\dots,2^{j}-1\} \mid \tilde{\mu}(I_{j,k})
\in [2^{-j(h+\varepsilon)},2^{-j(h-\varepsilon)}] \right\}  \\=
D(h).
\end{multline}
For $m \in \{0, \dots, N_T-1\}$, and $\varepsilon>0$, $\eta>0$,
denote by $A_j^{(m)}(\eta,\varepsilon)$ the event:
\begin{equation*}
\# \left\{ k \in \{0,\dots,2^{j}-1\} \mid \tilde{\mu}(I_{j,2^jm+k})
\in [2^{-j(h+\varepsilon)},2^{-j(h-\varepsilon)}] \right\} \ge
2^{j(D(h)-\eta)}.
\end{equation*}
Using the independence of the ${\cal M}$-cascades $(\mu^{(m)})_m$,
these events are independent and $\PP(A_j^{(m)}(\eta,\varepsilon))$
does not depend on $m$. Moreover, by \eqref{E:formalism_une}, for
any $\eta>0$, there exists $\epsilon>0$ such that
$\PP(A_j^{(m)}(\eta,\varepsilon)) \xrightarrow{j \to \infty}1$. We
easily deduce that for any $\eta>0$ and $\epsilon$ small enough:
\begin{equation*}
\mathop{\underline{\lim}}_j \frac{1}{N_T} \sum_{m=0}^{N_T-1} 1_{\{
A_j^{(m)}(\eta,\varepsilon) \}} \ge 1/2,~ \text{ almost surely.}
\end{equation*}
Since the cardinality of $\left\{k \in \{0,\dots,N-1 \} \mid
\tilde{\mu}(I_{j,k}) \in
[2^{-j(h+\varepsilon)},2^{-j(h-\varepsilon)}] \right\}$ is greater
than $2^{j(D(h)-\eta)} \sum_{m=0}^{N_T-1} 1_{\{
A_j^{(m)}(\eta,\varepsilon) \}}$ we deduce that the left hand side
of \eqref{E:res_inf_card} is greater than $D(h)+\chi-\eta$, for any
$\eta>0$. To end the proof, it suffices to show that the left hand
side of \eqref{E:res_sup_card} is lower than $D(h)+\chi$.
This is easily done, as in the end of the first step, by relying on
the asymptotic behavior of the partition function.
\end{proof}
\subsection{Central Limit Theorems}
\label{TCL}
In this section, we briefly study the rate of the convergence of
${\mathcal S}_{\tilde \mu}(j,p)$ as $j \rightarrow \infty$ of
\eqref{eq:Mpartitionscaling} in Theorem \ref{th:MSscaling}.
Using the same notations as in the proof of Theorem
\ref{th:MSscaling}, we write:
\begin{equation} \label{E:decompo_variance}
\mS(j,p)=\left\lfloor 2^{j\chi} \right\rfloor 2^{-j\tau(p)} \EE{
\mu_{\infty}([0,T])^p} + A_j+ B_j
\end{equation}
where
\begin{equation*}
A_j=\sum_{m=0}^{N_T-1} 2^{-jp} \sum_{r \in \{0,1\}^j}
\left(\prod_{i=1}^j W_{r|i}^{(m)p}\right)\big(\bar
\mu_\infty^{(m,r)}([0,T])^p-\EE{ \mu_{\infty}([0,T])^p} \big)
\end{equation*}
and
\begin{equation*}
B_j= \sum_{m=0}^{N_T-1} \big(2^{-jp} \sum_{r \in
\{0,1\}^j}\prod_{i=1}^j W_{r|i}^{(m)p}-2^{-j\tau(p)}\big)\EE{
\mu_{\infty}([0,T])^p}.
\end{equation*}
\begin{prop}
Assume \eqref{eq:hypo1}--\eqref{eq:momentsW} and that, either
$p_\chi^+=\infty$, or $p_\chi^+<\infty$ with $\tau(p_\chi^+)>0$. If
$p_\chi^- /2 < p <p_\chi^+/2$ then,
\begin{equation*}
2^{j(\tau(2p)-\chi)/2}A_j \xrightarrow{j \to \infty} {\mathcal
N}\big(0,\text{Var}\left(\mu_{\infty}([0,T])^p\right)\big).
\end{equation*}
\end{prop}
\begin{proof}
Consistently with the notations of Section \ref{S:prop_gene}, we
define, for every $r \in \{0,1\}^j$ and $m =0,\ldots, N_T-1$, the
random variables $\tilde{Z}^{(m,r)}=\bar
\mu_\infty^{(m,r)}([0,T])^p-\EE{ \mu_{\infty}([0,T])^p}$ and denote
by $ \tilde{Z}=\mu_{\infty}([0,T])^p-\EE{\mu_{\infty}([0,T])^p}$
their common law.
Furthermore, we will need the quantity
\begin{equation} \label{eta}
\eta_{m,r,j}(p)=2^{j (\tau(2p)-\chi)/2}2^{-jp}\left(\prod_{i=1}^j
W_{r|i}^{(m)p}\right) \tilde{Z}^{(m,r)},
\end{equation}
and the following family of $\sigma$-fields:
for $j \geq 0$
\begin{equation*}
{\mathcal F}_{-1,j}:=\sigma\big(W_r^{(m)},\;|r|\leq j,\;m=0,\ldots,
N_T-1\big)
\end{equation*}
and for every $k=0,\ldots, n(j)=2^j(N_T-1)$
\begin{equation*}
{\mathcal F}_{k,j}={\mathcal F}_{-1,j} \vee
\sigma\big(\tilde{Z}^{(m,r)},\;\bar r+2^jm \leq k\big).
\end{equation*}
For fixed $j$, we have a one-to-one correspondence between $(m,r)$
and $k=\bar r+2^jm$, so abusing notation slightly, we write
$\eta_{k,j}(p)$ instead of $\eta_{m,r,j}(p)$ in \eqref{eta} when no
confusion is possible. With these notations,
\begin{equation*}
2^{(\tau(2p)-\chi)/2} A_j=\sum_{k=0}^{n(j)} \eta_{k,j}(p)
\end{equation*}
where $\eta_{k,j}(p)$ is $\mathcal{F}_{k,j}$-measurable and
\begin{equation*} 
\EE{\eta_{k,j}(p)\,|\,{\mathcal F}_{k-1,j}}=0,\;\; \forall
k=0,\ldots, n(j).
\end{equation*}
Thus, we are dealing with a triangular array of martingale
increments. Let us consider the sum of the conditional variances:
\begin{equation}
V_j=\sum_{k=0}^{n(j)}\EE{\eta_{k,j}(p)^2\,|\,{\mathcal F}_{k-1,j}}.
\end{equation}
We have
\begin{equation*}
V_j =\text{Var}(\tilde{Z}) 2^{j(\tau(2p)-\chi)}
2^{-2jp}\sum_{m=0}^{N_T-1}\sum_{r \in \{0,1\}^j}\prod_{i=1}^j
\left(W_{r|i}^{(m)}\right)^{2p},
\end{equation*}
thus by application of Proposition \ref{P:cascade_conser_mixte}
(with the choice of $Z^{(m,r)}$ equal to $1$)
we get,
$$
V_j \xrightarrow{j \to \infty} \text{Var}(\tilde{Z}).
$$
Hence the proposition will be proved, if we can show that the
triangular array satisfies a Lindeberg condition: for some
$\epsilon>0$,
\begin{equation*}
V_j^{(\epsilon)}=\sum_{k=0}^{n(j)}\EE{\abs{\eta_{k,j}(p)}^{2+\epsilon}\,|\,{\mathcal
F}_{k-1,j}} \xrightarrow{j \to \infty} 0.
\end{equation*}
But, we have
\begin{equation*}
V_j^{(\epsilon)}=\EE{\abs{W}^{2+\epsilon}}
2^{j(\tau(2p)-\chi)(1+\epsilon/2)} 2^{-j(2+\epsilon)p} \sum_{r\in\{
0,1\}^j} \prod_{i=1}^j \left( W_{r|i}^{(m)} \right)^{2+\epsilon}
\end{equation*}
and by application of the  Proposition \ref{P:cascade_conser_mixte},
the order of magnitude of $V_j^{(\epsilon)}$ is
$2^{j \big( (\tau(2p)-\chi)(1+\epsilon/2)-(\tau((2+p)\epsilon) -
\chi)\big)}$. Thus, it can be seen that $V_j^{(\epsilon)}$ converges
to zero, for $\epsilon$ small enough, by the condition $ 2p
\tau'(2p)-\tau(2p)> -\chi$. This ends the proof of the proposition.
\end{proof}
\begin{prop} \label{T:TCL_B}
Assume \eqref{eq:hypo1}--\eqref{eq:momentsW} and that, either
$p_\chi^+=\infty$, or $p_\chi^+<\infty$ with $\tau(p_\chi^+)>0$.
Then:
\begin{enumerate}
\item If $\tau(2p)-2 \tau(p)>0$ we have,
\begin{equation} \label{E:TCL_A_normal}
2^{j(\tau(p)-\chi/2)} B_j \rightarrow {\mathcal N}\big(0, c(p)\big),
\end{equation}
where $c(p)>0$ depends on the law of $W$ and $p$.
\item
If $\tau(2p)-2 \tau(p)=0$, we have $\text{Var}(B_j)=O(j 2^{-j
(\tau(2p)-\chi)})$.
\item
If $\tau(2p)-2 \tau(p)<0$, we have $\text{Var}(B_j)=O(2^{-j
(\tau(2p)-\chi)})$.
\end{enumerate}
\end{prop}
\begin{proof}
Denote $\nu_j^{(m)}$ the measures defined at the step $j$ of the
construction of the {\cal M}-cascade on $[0,1]$ based on
$W^p/\EE{W^p}$:
\begin{equation*}
\nu_j^{(m)} ([0,1]) = 2^{-j} \sum_{r \in \{0,1 \}^j } \prod_{i=1}^j
(W^{(m,r)}_{r\mid i})^p \EE{W^p}^{-j}, \quad \text{ for $m \in
\{0,\dots, N_T-1 \}$.}
\end{equation*}
With this notation we have,
\begin{equation} \label{E:express_Bj}
B_j= \EE{ \mu_{\infty}([0,T])^p} 2^{-j \tau(p)} \sum_{m=0}^{N_T-1}
\left(\nu_j^{(m)} ([0,1])-1\right)
\end{equation}
and using Kahane and Peyri{\`e}re results \cite{KahPey76}, we know
that for each $m$ the sequence $(\nu_j^{(m)} ([0,1]))_j$ is bounded
in $\mathbf{L}^q$ as soon as $\tau(pq)-q \tau(p)>0$.

We first focus on the case $\tau(2p)-2\tau(p)>0$. Hence the sequence
$(\nu_j^{(m)}([0,1]))_j$ is bounded in
$\mathbf{L}^{2+\epsilon}$-norm for some $\epsilon >0$. Using that
$(\nu_j^{(m)}([0,1]) -1)_m$ is a centered i.i.d. sequence and
classical considerations for triangular array of martingale
increments, one can show that a central limit theorem holds:
\begin{equation*}
N_T^{-1/2} \sum_{m=0}^{N_T-1} (\nu_j^{(m)}([0,T]) -1) \xrightarrow{j
\to \infty} \mathcal{N}(0, \text{Var}(\nu_\infty^{(0)}([0,1]))),
\end{equation*}
where $\nu_\infty^{(0)}([0,1])=\lim_{j \to \infty}
\nu_j^{(0)}([0,1])$. From this and \eqref{E:express_Bj} we deduce
\eqref{E:TCL_A_normal} with $c(p)=\EE{\mu_{\infty}([0,T])^p}^2
\text{Var}(\nu_\infty^{(0)}([0,1]))$.

In the cases $\tau(2p)-2 \tau(p) \le 0$, by \eqref{E:express_Bj}
again we have
\begin{equation*}
\text{Var}(B_j)=\left\lfloor 2^{j \chi} \right\rfloor  2^{-j
2\tau(p)} \EE{\mu_{\infty}([0,T])^p}^2
\text{Var}(\nu_j^{(0)}([0,1])).
\end{equation*}
Now $\text{Var}(\nu_j^{(0)}([0,1]))= \EE{ \nu_j^{(0)}([0,1])^2 }-1$
is unbounded as $j \to \infty$, but a careful look at the
computations in Lemma \ref{lem:2} with $\epsilon=1$ yields to
$$
\EE{\nu_j^{(0)}([0,1])^2} \sim_{j \to \infty} \sum_{l=0}^j 2^{-l}
2^{l(2\tau(p)-\tau(2p))}.
$$
We deduce that $\text{Var}(B_j) =O\big( 2^{-j (2 \tau(p)-\chi)}
\sum_{l=0}^j 2^{-l} 2^{l(2\tau(p)-\tau(2p))}\big)$. Then, the
theorem follows in the cases $\tau(2p)-2 \tau(p) =0$ and $\tau(2p)-2
\tau(p) < 0$.
\end{proof}
\begin{rem}
By \eqref{E:decompo_variance} the difference between
$2^{(\tau(p)-\chi)j}\mS(j,p)$ and its limit is decomposed into two
dissimilar error terms: particularly the fact that the contribution
of $B_j$ converges to zero is due to the observation of a large
number of integral scales,
whereas the contribution of $A_j$
vanishes as the sampling step tends to zero.

In the case $\tau(2p)-2 \tau(p)>0$, the contribution of $B_j$
strictly dominates and
$2^{-(\tau(p)-\chi)}\mS(j,p)-\EE{\mu_{\infty}([0,T])^p}$ is of
magnitude $2^{-j \chi} \sim N_T^{-1/2}$.

If $\tau(2p)-2 \tau(p)<0$, the magnitude of $A_j$ and $B_j$ are the
same and $2^{-(\tau(p)-\chi)}\mS(j,p)-\EE{\mu_{\infty}([0,T])^p}$ is
asymptotically bounded by terms of magnitude $
2^{j/2(-\chi+2\tau(2p)-\tau(2p))}$. This rate of convergence is
slower than $N_T^{-1/2}$.
\end{rem}

\section{Extension to wavelet based partition functions}
\label{Sec:wavelets} The behavior of the partition function provides
an evaluation for the regularity of the sample path of the process
$t \mapsto \tilde{\mu}\left([0,t]\right)$. However, it is more
natural to assess this regularity via the behavior of wavelet
coefficients.

\subsection{Notations}
In this section we assume, for notational convenience, that $T=1$.
Consider now $g$ a ``generalized box'' function. It is a real valued
function that satisfies the following assumptions
\begin{itemize}
\item[(H1)]  $g$ has compact support included in $[0,2^J]$, for some $J \ge
0$.
\item[(H2)] $g$ is piecewise continuous.
\item[(H3)] $g$ is at least non zero on an interval.
\end{itemize}
Following the common wavelet notation, we define \begin{equation*}
g_{j,k}(t) = g(2^jt-k). \end{equation*}
 The support of $g_{j,k}(t)$
is
\begin{equation} \label{eq:supppsi} \mbox{Supp } g_{j,k} = [2^{-j}k
,2^{-j}k+2^{J-j}]. \end{equation} In the sequel, if $\mu$ is a
random measure, for any Borel function $f$
we will use the notation
\begin{equation*}
  \BB \mu f  = \int f(t) d\mu(t).
\end{equation*}


\subsection{The generalized partition function : scaling properties}
We define the generalized partition function of an {\cal M}-cascade
$\mu_{\infty}$ on $[0,1]$ at scale $2^{-j}$ as
\begin{equation}
\label{eq:wpart} \Sw(j,p) = \sum_{k=0}^{2^j-2^J-1}  |\BB
{\mu_{\infty}} {g_{j,k}}|^p .
\end{equation}
Remark that for simplicity we removed a finite number of border
terms, and that, in the case $g(t)$ is
the ``box'' function $g(t) = 1_{[0,1]}(t)$  we recover the partition
function of Section \ref{sec:mcascades}.

Let us study the scaling of $\EE{\Sw(j,p)}$.
\begin{prop}
\label{P:scaling_Sw} Assume \eqref{eq:hypo1}--\eqref{eq:momentsW}.
Then, we have $K_1 2^{-j \tau(p)} \le \EE{\Sw(j,p)} \le K_2 2^{-j
\tau(p)}$ for $K_1$, $K_2$, two positive constants depending on $p$,
$W$ and $g$.
\end{prop}
\begin{proof}
Since $|g(t)|$ is clearly a bounded function, we have
\begin{equation*}
\EE{|\BB {\mu_{\infty}} {g_{j,k}}|^p} \le C \EE{\mu_{\infty}([2^{-j}
k, 2^{-j} k + 2^{J-j}])^p},
\end{equation*}
where $C$ is a constant. We write $ \mu_{\infty}([2^{-j} k, 2^{-j} k
+ 2^{J-j}]) = \sum_{l=0}^{2^J-1} \mu_{\infty}(I_{j,k+l}),$
 and deduce
 \begin{equation}
\label{eq:gscaling} \EE{|\BB {\mu_{\infty}} {g_{j,k}}|^p} \le C
\EE{|\mu_{\infty}[0,  2^{-j}]  |^p} = K 2^{-j(\tau(p)+1)},
\end{equation} where $K$ only depends on $g$ and the law of $W$. By
\eqref{eq:wpart} we get the upper bound for $\EE{\Sw(j,p)}$.

For the lower bound, let us write that $\Sw(j,p)$ is greater than
$$
\sum_{k'=0}^{2^{j-J}-1}  |\BB {\mu_{\infty}} {g_{j,2^J k'}}|^p.
$$
But $g_{j,2^J k'}$ is supported on $[k' 2^{J-j},(k'+1)2^{J-j}]$,
thus applying Lemma \ref{L:scaling_mu} in the Appendix
\ref{app:lem34} with $a=j-J$, we deduce:
\begin{equation*}
\BB {\mu_{\infty}} {g_{j,2^J k'}}=2^{J-j} \left(  \prod_{i=1}^{j-J}
W_{r|i} \right) Z
\end{equation*}
where, in law, $Z$ is equal to $\BB{\mu_{\infty}}{g_{J,0}}$. Thus
$\EE{|\BB {\mu_{\infty}} {g_{j,k}}|^p}$ is greater than
$$
2^{p(J-j)} \EE{ W ^p}^{j-J} \EE{  | \BB{\mu_{\infty}}{g_{J,0}}|^p
}=K 2^{-j (\tau(p)+1)} \EE{  |\BB{\mu_{\infty}}{g_{J,0}}|^p }.
$$
Applying Lemma \ref{L:mu_f_pos} with $f=g_{J,0}$ shows that  $\EE{
|\BB{\mu_{\infty}}{g_{J,0}}|^p}$ is some positive constant. Then the
lower bound for $\EE{\Sw(j,p)}$ easily follows.
\end{proof}


\subsection{The partition function in the mixed asymptotic framework}
Following \eqref{eq:partitionmix1}, we define the partition function in the mixed asymptotic framework as
\begin{equation*}
\mSw(j,p)
 = \sum_{m=0}^{N_{T}-1}  \Sw^{(m)}(j,p),
 \end{equation*}
 where $\Sw^{(m)}(j,p)$ is the partition function of $\mu_{\infty}^{(m)}$, i.e.,
\begin{equation*}
 \Sw^{(m)}(j,p) = \sum_{k=0}^{2^j-2^J-1} \left|\BB {\mu_{\infty}^{(m)}} {g_{j,k}}\right|^p.
 \end{equation*}
We have the following result.
\begin{theo}
\label{th:wMSscaling}
{\bf Scaling of the generalized partition function in a mixed asymptotic } \\
Let $p >0$, then under the same assumptions as Theorem
\ref{th:MSscaling} the power law scaling of $\mSw(j,p)$ is given by
\begin{equation*}
\lim_{j \rightarrow \infty} \frac{\log_{2}\mSw(j,p)}{-j} \operatornamewithlimits{\longrightarrow}_{a.s.} \tau_{\chi}(p).
\end{equation*}
\end{theo}


\begin{proof}
Using Proposition \ref{P:scaling_Sw} we have,
 $\lim_{j
\rightarrow \infty} \frac{1}{-j} \log_{2}\EE{\mSw(j,p)} =
\tau_{\chi}(p), $
and
 we just need to prove that almost surely,
 $$
\mSw(j,p)-\EE{\mSw(j,p)}=o(2^{-j\tau_\chi(p)}).
 $$
Using lemma \ref{lem:3} and \ref{lem:4} of appendix \ref{app:lem34},
this is done in the exact same way as the proof of
\eqref{E:notre_but_cv} in Proposition \ref{P:cascade_conser_mixte}.
 \end{proof}


\section{Link with Besov spaces}

\label{S:Link_with}
Following \cite{Jaf00}, one may define for a measure $\mu$ on
$[0,T]$,  the boundary of its Besov domain as the function $s_\mu:
(0,\infty) \to \mathbb{R}\cup \{\infty\}$ given by
\begin{equation*}
s_\mu(1/p)=\sup \left\{ \sigma \in \mathbb{R} \mid \sup_{j \ge 0}
2^{j {\sigma}} \left( 2^{-j} \sum_{k=0}^{2^j-1} |  \mu(I_{j,k})|^{p}
\right)^{ 1/p} < \infty \right\}.
\end{equation*}
The following proposition can be shown (see \cite{Jaf00}).
\begin{prop} \label{P:prop_s} The function $s_\mu$ is an increasing, concave
function, with a derivative bounded by $1$.
\end{prop}
Let us stress that the condition $s'_\mu(1/p) \le 1$ is a simple
consequence of the Sobolev embedding for Besov spaces. The Theorem
\ref{th:Sscaling} characterizes the Besov domain for $\mu_{\infty}$
a {\cal M}-cascade on $[0,T]$:
\begin{equation*}
\forall p >0, \quad s_{\mu_{\infty}}(1/p)=
\begin{cases}
\frac{\tau(p)+1}{p} & \text{ if $\frac{1}{p}> \frac{1}{p_0}$}
\\
h_0^+ + \frac{1}{p} & \text{ if $\frac{1}{p}\le\frac{1}{p_0}$}
\end{cases}.
\end{equation*}
If we denote $s(\frac{1}{p})=\frac{\tau(p)+1}{p}$ then, it is simply
checked that the condition $1/p
>1  /p^+_0$ is equivalent to $s'(1/p)<1$. Hence Proposition
\ref{P:prop_s} explains why for $1/p \le 1/p^{+}_0$ the boundary of
the Besov domain must be linear with a slope equal to one.

In mixed asymptotic the support of the measure grows with $j$
but we can still define, using the notations of Section
\ref{Sec:mas}, the index:
\begin{equation*}
s^\chi_{\tilde\mu}(1/p)=\sup \left\{ \sigma \in \mathbb{R} \mid
\sup_{j \ge 0} 2^{j {\sigma}} \left( N_T^{-1} 2^{-j} \sum_{k=0}^{N_T
2^j-1} | \tilde{\mu}(I_{j,k})|^{p} \right)^{ 1/p} < \infty \right\}.
\end{equation*}
Then, it is simply checked that Theorem \ref{th:MSscaling} implies
$s^\chi_{\tilde\mu}(1/p)=s(1/p)$ when $s'(1/p) < 1+\chi$, and
$s^\chi_{\tilde\mu}(1/p)=h_\chi^+ + \frac{1+\chi}{p}$ otherwise.
This shows how the linear part in $s^\chi_{\tilde\mu}$ is shifted to
larger values of $p$ under the mixed asymptotic framework.

\section{Numerical examples and applications}
\label{Sec:num}

\begin{figure}
\includegraphics[height=10cm]{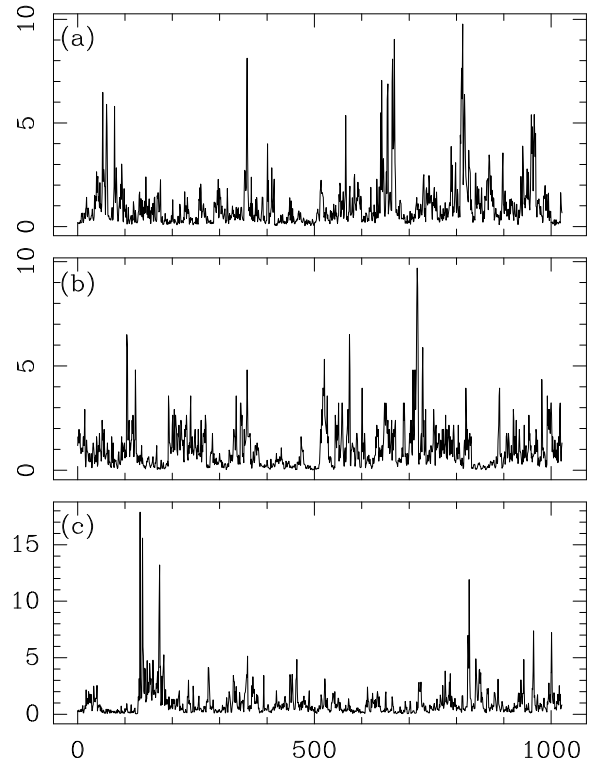}
\caption{\label{fig0} Three synthetic samples of $\cal M$-cascades
with $T = 2^{13}$ and $\lambda^2 = 0.2$: (a) log-Normal sample, (b)
log-Poisson sample with $\delta = -0.1$ and (c) log-Gamma sample
with $\beta = 10$. In fact we used ${\mu_{j_{\max}+5}[n,n+1]}_{n=0
\ldots L}$ as a proxy of $\mu_\infty[n,n+1]$ with $j_{\max} =
\log_2(T) = 13$ (see Eq. \eqref{eq:def_mul} for the definition of
$\mu_j$).}
\end{figure}

Our goal in this section is not to focus on statistical issues and
notably on precise estimates of multifractal exponents from
empirical data. We rather aim at illustrating the results of theorem
\ref{th:MSscaling} on precise examples, namely random cascades with
respectively log-normal, log-Poisson and log-Gamma statistics. For
the sake of simplicity we will consider exclusively scaling of
partition function for $p \geq 0$\footnote{Numerical methods for
estimating $\tau(p)$ for $p<0$ are trickier to handle}. In order to
facilitate the comparison of the three models, $\lambda^2$ will
represent the so-called intermittency coefficient, i.e.,
\begin{equation*}
  \lambda^2 = -\tau''(0)
\end{equation*}
where $\tau(p)$ is defined in Eq. \eqref{eq:taup}. This value will
be fixed for the three considered models. Let $\{W_r\}$ be the
cascade random generators as defined in Eq. \eqref{eq:def_mul}
and let $\omega_r = \ln W_r$.

In the simplest, log-Normal case the $\{\omega_r\}_r$ are normally
distributed random variables of variance $\lambda^2 \ln(2)$. Thanks
to the condition $\EE{W_r} = \EE{e^{\omega_r}} = 1$, their mean is
necessarily $-\lambda^2\ln(2)/2$. In that case, the cumulant
generating function $\tau(p)$ defined in Eq. \eqref{eq:taup} is
simply a parabola:
\begin{equation*}
  \tau^{n}(p) = p (1+\frac{\lambda^2}{2}) - \frac{\lambda^2}{2} p^2 -1
\end{equation*}

In the log-Poisson case, the variables $\omega_r$ are written as
$\omega_r = m_0 \ln(2) + \delta n_r$ where the $n_r$ are integers
distributed according to a Poisson law of mean $\gamma \ln(2)$. It
results that $\tau(p) = p(1-m_0) + \gamma(1-e^{p \delta}) -1$. If
one sets $\tau(1) = 0$ and $\tau''(0) = -\lambda^2$, one finally
gets the expression of $\tau(p)$ of a log-Poisson cascade with
intermittency coefficient $\lambda^2$:
\begin{equation}
\label{eq:taulp}
  \tau^{p}(p) = p  \left(1+\frac{\lambda^2}{\delta^2}(e^{\delta}-1)\right) + \frac{\lambda^2}{\delta^2}(1-e^{p \delta})
\end{equation}

In third case the variables $\omega_r$ are drawn from a
Gamma distribution. If $x$ is a random variable of pdf
$\beta^{\alpha \ln(2)} x^{\alpha \ln(2)-1} e^{-\beta x}
/\Gamma(\alpha \ln(2))$, then one chooses $\omega_r = x+ m_0 \ln(2)$
and it is easy to show that $\tau(p)$ is defined only for $p <
\beta$ and in this case $\tau(p) =  p(1-m_0)+\alpha (1-p/\beta)$. By
fixing $\tau(1) = 1$ and $\tau''(0) = \lambda^2$, on obtains:
\begin{equation}
\label{eq:taulg}
  \tau^{g}(p) = p \left(1-\lambda^2 \beta^2 \ln \frac{\beta-1}{\beta} \right) +
  \lambda^2 \beta^2 \ln \frac{\beta-p}{\beta}
\end{equation}
Notice that one recovers the log-normal case from both 
log-Poisson and log-Gamma statistics in the limits $\delta \rightarrow 0$ and
$\beta \rightarrow +\infty$ respectively.

For the 3 cases, one can explicitely compute all the mixed
asymptotic exponents as functions of $\chi$: In particular the
values of $p_\chi^{\pm}$ read:
\begin{eqnarray*}
   p^{\pm}_{\chi,n} & = & \pm \sqrt{\frac{2(1+\chi)}{\lambda^2}} \\
   p^{\pm}_{\chi,p} & = & \frac{W\left(\pm,\frac{\delta^2(1+\chi)-\lambda^2}{e \lambda^2}\right)+1}{\delta} = p^{\pm}_{\chi,n} + \frac{2(1+\chi)}{3\lambda^2}\delta + \mathcal{O}(\delta^2)\\
   p^{\pm}_{\chi,g} & = & \beta \left[1+\frac{1+\chi}{\lambda2 \beta^2}-e^{1+W(\pm,-e^{-1-\frac{1+\chi}{\lambda^2 \beta^2}})}   \right] = p^{\pm}_{\chi,n} - \frac{4(1+\chi)}{3\lambda^2 \beta} + \mathcal{O}(\beta^{-2})
\end{eqnarray*}
where suffixes $n,p,g$ stand for respectively log-normal,
log-Poisson and log-Gamma cascades and $W(\pm,z)$ represent the two
branches of the Lambert $W(z)$ function, solution of $W(z) e^{W(z)}
= z$ that take (respectively positive and negative) real values for
the considered arguments. For log-Poisson and log-Gamma cases, we
have also indicated the asymptotic behavior in the limits $\delta
\rightarrow 0$ and $\beta \rightarrow \infty$. The values
$h_\chi^{\pm}$ can be easily deduced form their definition:
$h_{\chi}^{\pm} = \tau'(p_\chi^{\pm})$.

In Fig. \ref{fig0} is plotted a sample of each of the three
examples of ${\cal M}$-cascades. We
chose $T = 2^{13}$ and $\lambda^2 = 0.2$ for all models while, in
the log-Poisson case we have set $\delta = -0.1$ and $\beta = 10$ in
the log-Gamma model. In each case, an approximation of the  ${\cal M}$-cascade sample is
generated. We chose to generate $\mu_{18}$ (as defined by  Eq. \eqref{eq:def_mul}) so that the smallest scale involved
is $l_{min} = 2^{-18}T = 2^{-5}$ (we have checked that the
results reported below do not depend on $l_{min}$). An approximation of $\tilde \mu$ is generated by concatenating i.i.d. realizations of
$\mu_{18}$.
Then, for each model and
for each chosen value of $\chi$,  $\tau_\chi(p)$ ($p =
0 \ldots 6$) was obtained from a least square fit of the curve
$\log_2 S_{\tilde \mu}(j,p)$ versus $j$ over the range $j = 0 \ldots
6$. Let us recall that,   for each value of $j$, the mixed asymptotic regime corresponds to
 sampling $\tilde \mu$  at scale  $l = 2^{-j}T$ and over an interval of size $L = 2^{j\chi}T$.
The exponents reported in figs. \ref{fig1} and \ref{fig2}
represent the mean values of exponents estimated in that way using
$N = 130$ experiments.

\begin{figure}
\includegraphics[height=6cm]{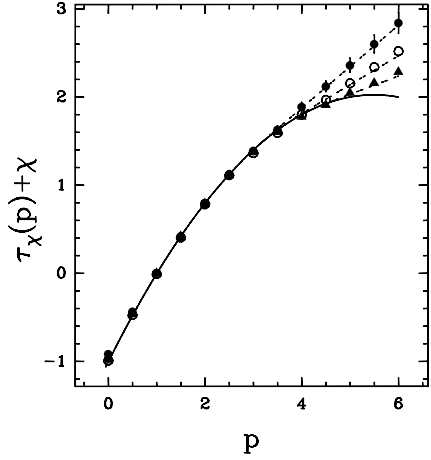}
\caption{\label{fig1} Estimates of the function $\tau_\chi(p)$ of
  log-normal cascades with $\lambda^2 = 0.2$ for $\chi = 0$ ($\bullet$)
$\chi = 0.5$ ($\circ$) and $\chi = 1$ ($\blacktriangle$). Dashed
lines represent the corresponding analytical expression from theorem
\ref{th:MSscaling} and the solid line represents the function
$\tau(q)$ as defined in Eq. \eqref{eq:taup}. $\tau_\chi$ is
estimated from the average over 130 trials of $2^{j\chi}$ cascades
samples. Error bars are reported on the $\chi = 0$ curve as vertical
solid bar. These errors are of order of symbol size.}
\end{figure}
\begin{figure}
\includegraphics[height=5cm]{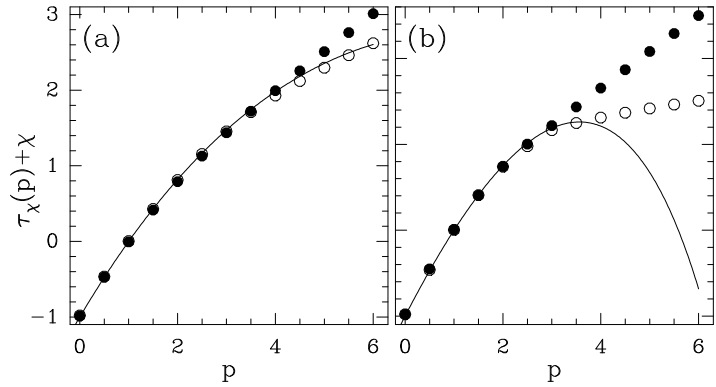}
\caption{\label{fig2} Estimates of the function $\tau_\chi(p)$ of
log-Poisson for $\chi = 0$ ($\bullet$) and $\chi = 1$ ($\circ$). In
both models we chose $T = 2^{13}$ and $\lambda^2 = 0.2$. (a)
Log-Poisson case with $\delta = -0.1$. (b) Log-Gamma case with
$\beta = 10$. Solid lines represent the curves $\tau(p)$.}
\end{figure}

The log-Normal mixed asymptotic scaling exponents for $\chi =
0,0.5,1$ are represented in Fig. \ref{fig1}. For illustration
purpose we have plotted $\tau_\chi(p) + \chi$ as a function of $p$:
one clearly observes that, as the value of $\chi$ increases, the
value of $p^+_\chi$ below which the function is linear, also
increases while the value of the slope $h_\chi^+$ decreases. As
expected, when $\chi$ increases $\tau_\chi(p)+\chi$ matches
$\tau(p)$ over an increasing range of $p$ values. Notice that the
estimated exponents are very close the analytical predictions as
represented by the dashed lines. Error bars on the mean value
estimates are simply computed from the estimated r.m.s. over the 130
trials and are reported only for the $\chi = 0$ curve. We can see 
that these errors are smaller or close to the symbol tickness.

In Fig. \ref{fig2} are reported estimates of $\tau_\chi(p)$, $\chi =
0,1$ for log-Poisson (fig. \ref{fig2}(a)) and log-Gamma (fig.
\ref{fig2}(b)) samples. The solide lines represent the theoretical
$\tau(p)$ functions for both models as provided by Eqs.
\eqref{eq:taulp} and \eqref{eq:taulg}. We used the same estimation
procedure as for the log-normal case. One sees that, in both cases,
since the intermittency coefficient is the same for the three
models, the classical $\tau_0(p)$ curves are very similar to the
log-normal curve (Fig. \ref{fig1}). However, these models behave
very differently in mixed regime: for $\chi = 1$, log-Poisson and
log-Gamma both estimated scaling exponents become closer to the
respective values of $\tau(p)$ and are very easy to distinguish.
Let us mention that such a analysis has been recently performed by two of us
in order to distinguish two popular log-normal and log-Poisson
models for spatial fluctuations of energy dissipation in fully
developed turbulence \cite{MuzBaBaPo07}.


\appendix

\section{Lemma used for the proof of Theorem \ref{th:MSscaling}}
\label{app:lem12}
\begin{lem}
\label{lem:1} We have
\begin{equation*}
L_{N}^{1+\epsilon} \le C 2^{j\chi} \EE{\abs{
\mathcal{M}_j^{(0)}(p)}^{1+\epsilon} }
\end{equation*}
where  $L_{N}^{1+\epsilon}$ is defined by \eqref{eq:step1} and $C$
is a constant that depends only on $\epsilon$.
\end{lem}
\begin{proof}
According to \cite{BahEss65}, if $\epsilon \in [0,1]$ and if
$\{X_{i}\}_{1\le i\le P}$ are centered independent random variables
one has
\begin{equation*} \EE{\abs{\sum_{i=1}^{P}
X_{i}}^{1+\epsilon}} \le C \sum_{i=1}^{P}
\EE{\abs{X_{i}}^{1+\epsilon}},
\end{equation*}
 where $C$ is a constant
that depends only on $\epsilon$ (and neither on the law of $X$ nor
on $P$). Applying it with $P = N_{T} =  \lfloor 2^{j\chi} \rfloor $
to the expression \eqref{E:def_Nj} of $\mathcal{N}_j(p)$, and using
the fact that the random variables $\{\mathcal{M}^{(m)}_j(p) \}_m$
defined by \eqref{E:def_Mj} are i.i.d, one gets
\begin{align*}
 L_{N}^{1+\epsilon} &\le C 2^{j\chi}  \EE{\left| \mathcal{M}_j^{(0)}(p) -
\EE{\mathcal{M}_j^{(0)}(p)}\right|^{1+\epsilon}}
\\ & \le C 2^{j\chi} \left(\EE{ \abs{ \mathcal{M}_j^{(0)}(p) }^{1+\epsilon} }+
\EE{\abs{\mathcal{M}_j^{(0)}(p)}}^{1+\epsilon}\right).
\end{align*}
Using the Jensen inequality we get the result.
\end{proof}
\begin{lem} Assume that $\EE{|Z|^{1+\epsilon}}<\infty$. Then we have for all $m$,
\label{lem:2}
\begin{equation*}
 \EE{ \abs{ \mathcal{M}^{(m)}_j(p) }^{1+\epsilon} } \le C 2^{-j(1+\epsilon)\tau(p)} \sum_{k=0}^{j} 2^{-k\tau(p(1+\epsilon))} 2^{k(1+\epsilon)\tau(p)},
\end{equation*}
 where $C$ is a constant that depends only on $p$ and
$\epsilon$.
\end{lem}
\begin{proof}
The proof of this result is very much inspired from
\cite{ResSamGilWil03}. Since the law of $\mathcal{M}^{(m)}_j(p)$ is
independent of $m$, we forget the supscript $m$ throughout the
proof. Using the definition \eqref{E:def_Mj}, one gets
\begin{equation}
\label{eq:interm}
   \EE{\abs{ \mathcal{M}_j(p)}^{1+\epsilon} } = 2^{-jp(1+\epsilon)}
\EE{\abs{\sum_{r\in\{0,1\}^j}  \prod_{i=1}^j W_{r|i}^p
Z^{(r)}}^{1+\epsilon}}.
\end{equation}
Let
\begin{equation}
\label{eq:X} X = 2^{-jp} \sum_{r\in\{0,1\}^j}  \prod_{i=1}^j
W_{r|i}^p Z^{(r)},
\end{equation}
then
\begin{equation*}
X^2 = 2^{-2jp} \sum_{r_{1}\in\{0,1\}^j}
\sum_{r_{2}\in\{0,1\}^j} \prod_{i=1}^j  W_{r_{1}|i}^pW_{r_{2}|i}^p
Z^{(r_1)} Z^{(r_2)}.
\end{equation*}
It can be rewritten as
\begin{equation}
\label{eq:X2bis}
X^2 = Y + D,
\end{equation}
where $Y$ corresponds to the non diagonal terms :
\begin{equation}
\label{eq:Y} Y = 2^{-2jp} \sum_{r_{1}\in\{0,1\}^j}
\sum_{\substack{r_{2}\in\{0,1\}^j\\ r_{2} \neq r_{1}}} \prod_{i=1}^j
W_{r_{1}|i}^pW_{r_{2}|i}^p Z^{(r_1)} Z^{(r_2)},
\end{equation}
and $D$ to the diagonal terms
\begin{equation}
\label{eq:D} D = 2^{-2jp} \sum_{r\in\{0,1\}^j} \prod_{i=1}^j
W_{r|i}^{2p}  \left(Z^{(r)}\right)^2.
\end{equation}
The left hand side of \eqref{eq:interm} is nothing but
$\EE{\abs{X}^{1+\epsilon}}$. By writing that
 $\EE{\abs{X}^{1+\epsilon}} =  \EE{\left(X^2\right)^{\frac{1+\epsilon}2}}$, using
 the sub-additivity of $x \mapsto x^{(1+\epsilon)/2}$,
 we get
 \begin{equation}
 \label{eq:Xeps}
 \EE{\abs{X}^{1+\epsilon}} \le \EE{|Y|^{\frac{1+\epsilon} 2}} + \EE{D^{\frac{1+\epsilon} 2}}.
 \end{equation}
Let us first work with the $Y$ term. We factorize the common
beginning of the words $r_1$ and $r_2$ in the expression
\eqref{eq:Y} of $Y$
\begin{equation*} Y = 2^{-2jp} \sum_{k=0}^{j-1}
\sum_{r\in\{0,1\}^k} \prod_{i=1}^{k} W_{r|i}^{2p} \sum_{\substack
{r_1,r_2\in\{0,1\}^{j-k} \\ r_1|0 \neq r_2|0}}~ \prod_{i=k+1}^j
W_{rr_{1}|i}^pW_{rr_{2}|i}^p Z^{(r_1)} Z^{(r_2)}.
\end{equation*}
Again by the sub-additivity of $x \mapsto x^{(1+\epsilon)/2}$ and
using the fact that the  $W_{r|i}$ are i.i.d., one gets
\begin{eqnarray*}
  \EE{|Y|^{\frac {1+\epsilon} 2}}&  \le  &
 2^{-jp(1+\epsilon)} \sum_{k=0}^{j-1} \EE{W^{p(1+\epsilon)}}^k   \sum_{r\in\{0,1\}^k}
 \\
  &
  &
 \EE{
\Big( \sum_{\substack {r_1,r_2\in\{0,1\}^{j-k} \\ r_1|0 \neq
r_2|0}}~ \prod_{i=k+1}^j
 W_{rr_{1}|i}^pW_{rr_{2}|i}^p
|Z^{(r_1)} Z^{(r_2)}| \Big)^{(1+\epsilon)/2}},
\end{eqnarray*}
and by using Jensen inequality
\begin{eqnarray}
\nonumber
  \EE{|Y|^{\frac {1+\epsilon} 2}} & \le &
2^{-jp(1+\epsilon)} \sum_{k=0}^{j-1}
\EE{W^{p(1+\epsilon)}}^k
\sum_{r\in\{0,1\}^k} \\
\label{eq:interm3}  & &\Big( \sum_{\substack
{r_1,r_2\in\{0,1\}^{j-k} \\ r_1|0 \neq r_2|0}}~ \prod_{i=k+1}^j
\EE{W_{rr_{1}|i}^p W_{rr_{2}|i}^p} \EE{ |Z^{(r_1)} Z^{(r_2)}| }
\Big)^{(1+\epsilon)/2}.
\end{eqnarray}
The variables $Z^{(r_1)}$ and $Z^{(r_2)}$ are independant with
finite expectation, thus the term
 $\EE{|Z^{(r_1)} Z^{(r_2)}| }$ is bounded by a constant $C$.
Using $\prod_{i=k+1}^j \EE{W_{rr_{1}|i}^pW_{rr_{2}|i}^p} =
\EE{W^p}^{2(j-k)}$, we deduce:
\begin{eqnarray*}
 \EE{|Y|^{\frac{1+\epsilon} 2}} &\le&   C 2^{-jp(1+\epsilon)}
 \sum_{k=0}^{j-1} \EE{W^{p(1+\epsilon)}}^k
\EE{W^p}^{(j-k)(1+\epsilon)} \\
\nonumber
&& \sum_{r\in\{0,1\}^k}
|
\sum_{\substack {r_1,r_2\in\{0,1\}^{j-k} \\ r_1|0 \neq r_2|0}}~ 1
|^{(1+\epsilon)/2}.
\end{eqnarray*}
There are $2^k$ possible values for $r$ and less than $2^{2(j-k)}$ values for the couple $(r_1,r_2)$, thus
\begin{equation*}
 \EE{|Y|^{\frac{1+\epsilon} 2}} \le
  2^{-jp(1+\epsilon)} K
 \sum_{k=0}^{j-1} \EE{W^{p(1+\epsilon)}}^k
\EE{W^p}^{(j-k)(1+\epsilon)} 2^k 2^{(j-k)(1+\epsilon)}.
\end{equation*}
Since $2^{-j\tau(p)} = 2^{-jp} 2^{j} \EE{W^p}^j$
\begin{equation}
\label{eq:Yfinal}
 \EE{\abs{Y}^{\frac{1+\epsilon} 2}}  \le K 2^{-j(1+\epsilon)\tau(p)} \sum_{k=0}^{j-1} 2^{-k\tau(p(1+\epsilon))} 2^{k(1+\epsilon)\tau(p)}.
 \end{equation}
 Let us now take care of the diagonal terms  of $X$ \eqref{eq:D}.
 First, we write
 \begin{equation*}
 D^{\frac{1+\epsilon} 2} \le
2^{-jp(1+\epsilon)} \sum_{r\in\{0,1\}^j} \prod_{i=1}^j
W_{r|i}^{p(1+\epsilon)} |Z^{(r)}|^{1+\epsilon},
 \end{equation*}
 and using the $\EE{\abs{Z}^{1+\epsilon}}< \infty$ we deduce that
 \begin{equation}
\label{eq:Dfinal}
 \EE{D^{\frac{1+\epsilon} 2}} \le C
2^{-jp(1+\epsilon)} 2^j  \EE{W^{p(1+\epsilon)}}^j  = C 2^{-j\tau(p(1+\epsilon))}.
 \end{equation}
 Merging \eqref{eq:Yfinal} and \eqref{eq:Dfinal} into \eqref{eq:Xeps} leads to
 \begin{equation*}
 \EE{\abs{X}^{1+\epsilon}}  \le K 2^{-j(1+\epsilon)\tau(p)} \sum_{k=0}^{j} 2^{-k\tau(p(1+\epsilon))} 2^{k(1+\epsilon)\tau(p)},
 \end{equation*}
 and since $\EE{\abs{\mathcal{M}_j(p)}^{1+\epsilon} }  =  \EE{ \abs{X}^{1+\epsilon}} $, it completes the proof.
\end{proof}

\section{Lemma used for the proof of Theorem \ref{th:wMSscaling}}
\label{app:lem34}
\begin{lem} \label{L:scaling_mu} Let $f: [0,1] \to \R$ be some
Borel function whose support is included in
$I_{a,\overline{r}}=[\frac{{\overline{r}}}{2^a},\frac{({\overline{r}}+1)}{2^a}
]$ for $r \in \{0,1 \}^a$, $a \ge 0$. Then
\begin{equation} \label{E:scaling_f_mu}
\BB{\mu_{\infty}}{f}=2^{-a} \left( \prod_{i=1}^a W_{r \mid i}
\right)
 \BB{\overline{\mu}_{\infty}^{(r)}}{ \tilde{f}}
\end{equation}
where $\tilde{f}(x)=f(2^{-a}(x+\overline{r}))$ and
$\overline{\mu}_{\infty}^{(r)}$ is a cascade measure on $[0,1]$
measurable with respect to the sigma field $\sigma \{ W_{rr'}, r'
\in \{0,1\}^{a'}, a' \ge 1 \}$.
\end{lem}
\begin{proof} The scaling relation \eqref{E:scaling_f_mu} is
easily obtained, by the definition of the measure $\mu_\infty$, if
$f$ is the characteristic function of some interval
$I_{a+a',\overline{rr'}}$ where $r' \in \{0,1\}^{a'}$, $a' \ge 0$.
This relation extends to any Borel function $f$ by standard
arguments of measure theory.
\end{proof}
\begin{lem} \label{L:mu_f_pos} Let $h: [0,1] \to \R$ be a piecewise
continuous, non zero, function. Then
$\EE{\abs{\BB{\mu_{\infty}}{h}}^p}> 0$ for all $p>0$.
\end{lem}
\begin{proof}
By  contradiction, assume that for some $p>0$,
$\EE{\abs{\BB{\mu_{\infty}}{h}}^p}=0$. Hence
$\BB{\mu_{\infty}}{h}=0$, $\PP$-almost surely. But using Lemma
\ref{L:scaling_mu},
\begin{align*}
0=\BB{\mu_{\infty}}{h}&=  \BB{\mu_{\infty}}{h 1_{[0, 1/2]}}  +
\BB{\mu_{\infty}}{h 1_{(1/2, 1]}}
\\
&= \frac{1}{2} W_0 \BB{\overline{\mu}^{(0)}_{\infty}}{h^{(0)}} +
\frac{1}{2} W_1 \BB{\overline{\mu}^{(1)}_{\infty}}{h^{(1)}},
\end{align*}
where $h^{(0)}(\cdot)=h(2^{-1} \cdot )$, $h^{(1)}(\cdot)=h(2^{-1}
(\cdot +1))$ and $\overline{\mu}^{(0)}_{\infty}$,
$\overline{\mu}^{(1)}_{\infty}$ are independent cascade measures on
$[0,1]$.
 Thus we deduce $ W_0
\BB{\overline{\mu}^{(0)}_{\infty}}{h^{(0)}} =- W_1
\BB{\overline{\mu}^{(1)}_{\infty}}{h^{(1)}}$ almost surely, and
since $W>0$ this shows that with probability one the two independent
variables $\BB{\overline{\mu}^{(0)}_{\infty}}{h^{(0)}}$ and
$\BB{\overline{\mu}^{(1)}_{\infty}}{h^{(1)}}$ vanish simultaneously.
This is only possible either, if they both vanish on a set of full
probability, or if they both vanish on a negligible set. Assume the
latter, then the following identity holds almost surely
\begin{equation*}
\frac{\BB{\overline{\mu}^{(0)}_{\infty}}{h^{(0)}}}{\BB{\overline{\mu}^{(1)}_{\infty}}{h^{(1)}}}
=- \frac{W_1}{W_0}
\end{equation*}
where the variables on right and left hand side are independent.
These variables must be constant, which is excluded by the
assumption $\PP(W=1)<1$ (recall \eqref{eq:hypo1}).

Thus we deduce that the variables
$\BB{\overline{\mu}^{(i)}_{\infty}}{h^{(i)}}$ are almost surely
equal to zero. Hence:
$$\EE{ \abs{ \BB{
{\mu}_{\infty}}{h^{(i)} }}^p} =0, \text{ for $i=0,1$}.$$ Iterating
the argument we deduce the following property:  for any $j \ge 0$
and $k \le 2^j-1$, if we define a function on $[0,1]$ by
$h^{(j,k)}(x)=h(2^{-j}(x+k))$ we have
$$\EE{
\abs{ \BB{ {\mu}_{\infty}}{h^{(j,k)}} }^p }=0.$$ This is clearly
impossible if we choose $j,k$ such that $h$ remains positive (or
negative) on $[k2^{-j}, (k+1)2^{-j}]$. By the assumptions on $h$ one
can find such an interval, yielding to a contradiction.
\end{proof}
\begin{lem} \label{lem:3} We have
\begin{equation*} \EE{ \abs{\mSw -\EE{\mSw}}^{1+\epsilon} } \le C 2^{j\chi}
\EE{\abs{\Sw(j,p)}^{1+\epsilon} }
 \end{equation*}
 where $C$ is a constant that depends only
on $\epsilon$.
\end{lem}
\begin{proof}
The proof is the exact same proof as for Lemma \ref{lem:1}.
\end{proof}
\begin{lem}
\label{lem:4} For any $\epsilon>0$ small enough we have,
\begin{equation*}
 \EE{\abs{\Sw(j,p) }^{1+\epsilon}} \le K
2^{-j(1+\epsilon)\tau(p)} \sum_{k=0}^{j} 2^{-k\tau(p(1+\epsilon))}
2^{k(1+\epsilon)\tau(p)},
\end{equation*}
 where $K$ is a constant
that depends only on $p$ and $\epsilon$.
\end{lem}
\begin{proof}
The proof basically follows the same lines as the proof of Lemma
\ref{lem:2}. The only difficulty, compared to this latter proof,
comes from the fact that the quantity $\BB{\mu_{\infty}}{g_{j,k}}$ a
priori involves several nodes of level $j$ of the {\cal M}-cascade.
We have to reorganize the sum \eqref{eq:wpart}.

Since we are interested in the limit $j\rightarrow +\infty$, we can
suppose, with no loss of generality that $j>J$. In the following, we
note $1^{(n)}$  the $n$-uplet
\begin{equation*} 1^{(n)} =
11\ldots1,~~\mbox{where the 1 is repeated $n$ times}.
\end{equation*}
The partition function \eqref{eq:wpart} can be
written $\Sw(j,p) = \sum_{k} |\BB {\mu_{\infty}} {g_{j,\overline
k}}|^p, $ where the sum is over $k\in \{0,1 \}^j$ such that $k_l=0$
for some $l \le j-J$. The sum can be regrouped in the following way,
where $a+1$ denotes the position of the last $0$ in the $j-J$ first
components of $k$:
 \begin{equation*} \Sw(j,p) =
\sum_{a=0}^{j-J-1} \sum_{\substack{r\in \{0,1\}^{a} \\ q = r0}}
\sum_{s\in \{0,1\}^J} \left| \BB {\mu_{\infty}} {g_{j,
\overline{q1^{(j-J-1-a)}s}}} \right|^p.
\end{equation*}
We set
\begin{equation} \label{E:def_Xsa} X_{a,s} =
\sum_{\substack{r\in \{0,1\}^{a} \\ q = r0}} \left| \BB
{\mu_{\infty}} {g_{j, \overline{q1^{(j-J-1-a)}s}}} \right|^p
\end{equation}
and consequently
\begin{equation*}
\Sw(j,p) = \sum_{a=0}^{j-J-1} \sum_{s\in \{0,1\}^J} X_{a,s}.
\end{equation*}
Actually, $a$ exactly corresponds to the level of the ``highest''
node that is common for dyadic intervals in the support of $g_{j,
\overline{q1^{(j-J-1-a)}s}}$. Indeed, let us prove that
\begin{equation} \label{eq:supincl} a\ge 0,~\forall
s\in\{0,1\}^J,~~\mbox{Supp } g_{j, \overline{q1^{(j-J-1-a)}s}}
\subset I_{a,\overline r},
\end{equation}
 where $q =r 0$. Indeed,
according to \eqref{eq:supppsi}, the support of $g_{j,
\overline{q1^{(j-J-1-a)}s}}$ is  included in $
[2^{-j}\overline{q1^{(j-J-1-a)}s},
2^{-j}\overline{q1^{(j-J-1-a)}s}+2^{-(j-J)}].$
Then,
\begin{eqnarray*}
2^{-j}\overline{q1^{(j-J-1-a)}s} & = &  2^{-a} \overline{r} + \sum_{i=a+2}^{j-J} 2^{-i} + 2^{-j} \overline{s} \\
& = & 2^{-a} \overline{r} +  2^{-a-1} - 2^{-(j-J)} + 2^{-j} \overline{s}.
\end{eqnarray*}
Since $\overline s$ varies in $[0,2^J-1]$, and $a\le j-J-1$, it is easy to show that
\begin{equation*}
0\le 2^{-a-1} - 2^{-(j-J)} + 2^{-j} \overline{s},
\end{equation*}
and
\begin{equation*}
2^{-a-1} + 2^{-j} \overline{s} \le 2^{-a},
\end{equation*}
which proves \eqref{eq:supincl}.

\noindent We are now ready to compute the upper bound for $\norme{
\Sw(j,p)}_{L^{1+\epsilon}(\PP)} \le \sum_{a=0}^{j-J-1} \sum_{s\in
\{0,1\}^J}  \norme{ X_{s,a} }_{L^{1+\epsilon}(\PP)}$. Using
\eqref{E:def_Xsa}, \eqref{eq:supincl} and Lemma \ref{L:scaling_mu},
we get
\begin{equation*} X_{a,s} = 2^{-ap} \sum_{r\in\{0,1\}^a}
  \left( \prod_{i=1}^{a}
W_{r|i}^p \right)
|\BB{\overline{\mu}_{\infty}^{(r)}}{g_{j,\overline{01^{(j-J-1-a)}s}}}|^p,
\end{equation*}
where the $\overline{\mu}_{\infty}^{(r)}$ are independent cascade
measures on $[0,T]$.

Let us identify $X_{a,s}$ with $X$ as defined in Lemma \ref{lem:2}
by \eqref{eq:X} in which $j$ plays the role of $a$ and $Z^{(r)}$ of
$\BB{\overline{\mu}_{\infty}^{(r)}}{g_{j-a,\overline{01^{(j-J-1-a)}s}}}^p$.
As in \eqref{eq:X2bis}, we can decompose $X_{a,s}^2$ as the sum of
the non diagonal terms $Y_{a,s}$ and the diagonal terms $D_{a,s}$
\begin{equation} \label{E:eq_X2as}
X_{a,s}^2 = Y_{a,s} + D_{a,s}.
\end{equation}
Using the exact same development as the one we used  for
$\EE{X^{1+\epsilon}}$ starting at Eq. \eqref{eq:interm}, we get the
bound of the non diagonal terms corresponding to \eqref{eq:interm3}
in which the term $\EE{ Z^{(rr_{1})}Z^{(rr_{2})} } $ has to be
replaced by
\begin{equation*}
 \EE{
\left|\BB{\mu_{\infty}^{(rr_{1},s)}}{g_{j-a,\overline{01^{(j-J-1-a)}s}}}\right|^p
\left|\BB{\mu_{\infty}^{(rr_{2},s)}}{g_{j-a,\overline
{01^{(j-J-1-a)}s}}}\right|^p}
\end{equation*}
 which can be bounded
(using \eqref{eq:gscaling}) by $K2^{-2(j-a)(\tau(p)+1)}$. Going on
with the same arguments as in Lemma \ref{lem:2}, we finally get, the
bound for the non diagonal terms corresponding to \eqref{eq:Yfinal}
\begin{align*}
 \EE{\abs{Y_{a,s}}^{\frac{1+\epsilon} 2}}
 &\le K 2^{-j(1+\epsilon)\tau(p)} 2^{(a-j)(1+\epsilon)} \sum_{k=0}^{a-1} 2^{-k\tau(p(1+\epsilon))} 2^{k(1+\epsilon)\tau(p)}
 \\
& \le K 2^{-j(1+\epsilon)\tau(p)} 2^{(a-j)(1+\epsilon)}
\sum_{k=0}^{j-1} 2^{-k\tau(p(1+\epsilon))} 2^{k(1+\epsilon)\tau(p)}
 .
\end{align*}
 Following the arguments in Lemma \ref{lem:2} for the diagonal
terms, we get
\begin{equation} \label{E:borne_D}
\EE{ D_{a,s}^{\frac{1+\epsilon} 2}} \le 2^{-j\tau(p(1+\epsilon))}
2^{a-j}.
\end{equation}
By \eqref{E:eq_X2as}--\eqref{E:borne_D},  we finally get
\begin{equation*}
 \EE{\abs{X_{a,s}}^{1+\epsilon}} \le  K  2^{(a-j)} 2^{-j(1+\epsilon)\tau(p)}
\sum_{k=0}^{j} 2^{-k\tau(p(1+\epsilon))} 2^{k(1+\epsilon)\tau(p)}.
\end{equation*}
Then we write
 \begin{align*} \EE{\abs{\Sw(j,p)}^{1+\epsilon} } \le \norme{
\Sw(j,p)}_{L^{1+\epsilon}(\PP)}^{1+\epsilon} \le \left(
\sum_{a=0}^{j-J-1} \sum_{s\in \{0,1\}^J}  \norme{ X_{a,s}
}_{L^{1+\epsilon}(\PP)}  \right)^{1+\epsilon}
\end{align*}
and the lemma follows.
\end{proof}

\bibliography{paper_ver5-3_IMS}

\end{document}